\documentclass[11pt]{amsart}
\usepackage[margin=1in]{geometry}

\usepackage{amssymb}
\usepackage{amsthm}
\usepackage{amsmath}
\usepackage{mathrsfs}
\usepackage{amsbsy}
\usepackage[all]{xy}
\usepackage{bm}
\usepackage{hyperref}
\usepackage{tikz}
\usepackage{array}
\usepackage{float}
\usepackage{enumerate}
\usepackage{xcolor}
\usepackage{hhline}
\allowdisplaybreaks
\usepackage{cite}

\newcommand{\tor}{\operatorname{tor}}
\newcommand{\Acyc}{\operatorname{Acyc}}
\newcommand{\sgn}{\operatorname{sgn}}
\newcommand{\tripprox}{\approx_{\mathrm{loc}}}
\DeclareMathOperator{\FS}{\mathsf{FS}}

\newenvironment{enumerate*}%
  {\begin{enumerate}[(I)]%
    \setlength{\itemsep}{10pt}%
    \setlength{\parskip}{0pt}}%
  {\end{enumerate}}

\newtheorem{theorem}{Theorem}[section]
\newtheorem{proposition}[theorem]{Proposition}
\newtheorem{corollary}[theorem]{Corollary}
\newtheorem{conjecture}[theorem]{Conjecture}

\newtheorem{lemma}[theorem]{Lemma}

\theoremstyle{definition}

\newtheorem{remark}[theorem]{Remark}
\newtheorem{example}[theorem]{Example}

\newcommand{\dfn}[1]{\textcolor{blue}{\emph{#1}}}

\begin{document}

\title[]{Friends and strangers walking on graphs}
\subjclass[2010]{}

\author[]{Colin Defant}
\address[]{Department of Mathematics, Princeton University, Princeton, NJ 08540, USA}
\email{cdefant@princeton.edu}
\author[]{Noah Kravitz}
\address[]{Department of Mathematics, Princeton University, Princeton, NJ 08540, USA}
\email{nkravitz@princeton.edu}

\begin{abstract} 
Given graphs $X$ and $Y$ with vertex sets $V(X)$ and $V(Y)$ of the same cardinality, we define a graph $\FS(X,Y)$ whose vertex set consists of all bijections $\sigma\colon V(X)\to V(Y)$, where two bijections $\sigma$ and $\sigma'$ are adjacent if they agree everywhere except for two adjacent vertices $a,b \in V(X)$ such that $\sigma(a)$ and $\sigma(b)$ are adjacent in $Y$. This setup, which has a natural interpretation in terms of friends and strangers walking on graphs, provides a common generalization of Cayley graphs of symmetric groups generated by transpositions, the famous $15$-puzzle, generalizations of the $15$-puzzle as studied by Wilson, and work of Stanley related to flag $h$-vectors. We derive several general results about the graphs $\FS(X,Y)$ before focusing our attention on some specific choices of $X$. When $X$ is a path graph, we show that the connected components of $\FS(X,Y)$ correspond to the acyclic orientations of the complement of $Y$. When $X$ is a cycle, we obtain a full description of the connected components of $\FS(X,Y)$ in terms of toric acyclic orientations of the complement of $Y$. We then derive various necessary and/or sufficient conditions on the graphs $X$ and $Y$ that guarantee the connectedness of $\FS(X,Y)$.  Finally, we raise several promising further questions.
\end{abstract}

\maketitle

\section{Introduction}\label{Sec:Intro}

Let $X$ be a simple graph with $n$ vertices. Imagine that $n$ different people, any two of whom are either friends or strangers, are standing so that one person is at each vertex of $X$.  At each point in time, two friends standing at adjacent vertices of $X$ may swap places by simultaneously crossing the edge that connects their respective vertices; two strangers may not swap places in this way. Our goal is to understand which configurations of people can be reached from other configurations when we allow the people to swap places repeatedly in this manner. The resolution of this problem certainly depends on the graph $X$, as well as on the information about which people are friends with each other; this further information can be encoded in a graph $Y$ whose vertex set is the set of people and where edges indicate friendship.

To formalize and refine this problem, we define the \dfn{friends-and-strangers graph} $\FS(X,Y)$ whose vertex set is the set of bijections $\sigma\colon V(X)\to V(Y)$. Two bijections $\sigma,\sigma'\colon V(X)\to V(Y)$ are adjacent in $\FS(X,Y)$ if and only if we can find an edge $\{a,b\}$ in $X$ such that:
\begin{itemize}
\item $\{\sigma(a),\sigma(b)\}$ is an edge in $Y$;
\item $\sigma(a)=\sigma'(b)$ and $\sigma(b)=\sigma'(a)$;
\item $\sigma(c)=\sigma'(c)$ for all $c\in V(X)\setminus\{a,b\}$.
\end{itemize}
When this is the case, we refer to the operation that transforms $\sigma$ into $\sigma'$ as an \dfn{$(X,Y)$-friendly swap across $\{a,b\}$}. Performing an $(X,Y)$-friendly swap corresponds to allowing two friends to swap places in the graph $X$. Notice that the isomorphism type of $\FS(X,Y)$ depends only on the isomorphism types of $X$ and $Y$. Since we will usually be concerned only with the graph-theoretic structure of $\FS(X,Y)$ (such as the number and sizes of connected components), we will often specify the graphs $X$ and $Y$ only up to isomorphism. 

It is sometimes convenient to assume that $V(X)$ and $V(Y)$ are both the set $[n]:=\{1,\ldots,n\}$. In this case, the vertices of $\FS(X,Y)$ are the elements of the symmetric group $\mathfrak S_n$, which consists of all permutations of the numbers $1,\ldots,n$. For $i,j\in[n]$, let $(i\,j)$ be the transposition in $\mathfrak S_n$ that swaps the numbers $i$ and $j$. If $\sigma\in\mathfrak S_n$ is such that $\{i,j\}$ is an edge in $X$ and $\{\sigma(i),\sigma(j)\}$ is an edge in $Y$, then we can perform an $(X,Y)$-friendly swap across $\{i,j\}$ to change $\sigma$ into the permutation $\sigma\circ(i\,j)$.  If we write the permutation $\sigma$ in one-line notation as $\sigma=\sigma(1)\cdots \sigma(n)$, then an $(X,Y)$-friendly swap transposes two entries of the permutation such that the positions of the entries are adjacent in $X$ and the entries themselves are adjacent in $Y$.

\begin{example}
If $\displaystyle X=Y=\begin{array}{l}\includegraphics[height=.5cm]{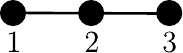}\end{array},\text{ then }\FS(X,Y)=\begin{array}{l}\includegraphics[height=1.3cm]{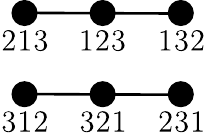}\end{array}$. 
\end{example}

\begin{example}\label{Exam2}
If $X=\begin{array}{l}\includegraphics[height=1.5cm]{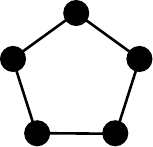}\end{array}\quad\text{and}\quad Y=\begin{array}{l}\includegraphics[height=1.7cm]{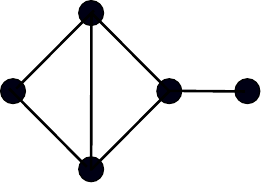}\end{array}$, then $\FS(X,Y)$ is the graph shown (without vertex labels) in Figure~\ref{Fig5}. Notice that $\FS(X,Y)$ has $5$ pairwise isomorphic connected components, each of which exhibits a highly symmetric structure; our results in Sections~\ref{Sec:General} and~\ref{Sec:Cycles} will explain both of these properties.
\end{example}

\begin{figure}[ht]
\[\begin{array}{l}\includegraphics[height=5cm]{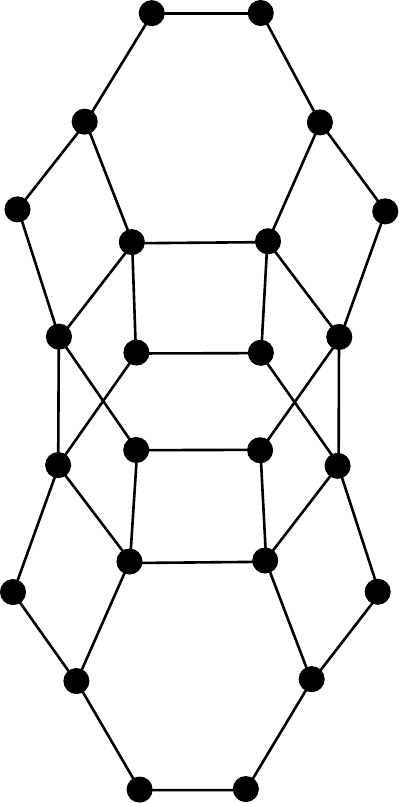}\end{array} \begin{array}{l}\includegraphics[height=5cm]{FriendsPIC16}\end{array} \begin{array}{l}\includegraphics[height=5cm]{FriendsPIC16}\end{array} \begin{array}{l}\includegraphics[height=5cm]{FriendsPIC16}\end{array} \begin{array}{l}\includegraphics[height=5cm]{FriendsPIC16}\end{array}\] 
\caption{The graph $\FS(X,Y)$ from Example~\ref{Exam2}.}
\label{Fig5}
\end{figure}

In this paper, we concern ourselves mainly with the problem of understanding the structure of $\FS(X,Y)$ for various graphs $X$ and $Y$. This is a very general setup, of which some previously-considered problems are special cases. One motivation for phrasing the problem in terms of permutations is to encompass previous work on multiplication by restricted transpositions. To give some examples, we let $K_n$ denote the complete graph on the vertex set $[n]$ and let $\mathsf{Path}_n$ be the path graph on the vertex set $[n]$ with edges $\{i,i+1\}$ for all $i\in[n-1]$. The graph $\FS(K_n,K_n)$ is the Cayley graph of $\mathfrak S_n$ generated by all transpositions. The graph $\FS(K_n,\mathsf{Path}_n)$ is the weak Bruhat graph of $\mathfrak S_n$, which is also the $1$-skeleton of the standard $(n-1)$-dimensional permutohedron. More generally, if $Y$ is connected, then $\FS(K_n,Y)$ is the Cayley graph of $\mathfrak S_n$ generated by the transpositions corresponding to the edges of $Y$. Letting $\overline Y$ denote the complement of $Y$, we find that $\FS(\mathsf{Path}_n,\overline Y)$ is the same as the update graph introduced by Barrett and Reidys in their work \cite{SDS} on sequential dynamical systems. 
Stanley's article \cite{stanley} focuses on the connected components of $\FS(\mathsf{Path}_n,\mathsf{Path}_n)$. 

Another special case of our setup has previously been phrased in terms of sliding tiles on graphs. Analyzing the famous 15-puzzle is equivalent to analyzing the connected components of $\FS(\mathsf{Grid}_{4\times 4}, \mathsf{Star}_{16})$, where $\mathsf{Grid}_{4\times 4}$ is a $4\times 4$ grid graph and $\mathsf{Star}_n$ denotes the star graph with vertex set $[n]$ and edges $\{i,n\}$ for all $i\in[n-1]$. The purpose of Wilson's article \cite{wilson} is to generalize the 15-puzzle by computing the number of connected components of $\FS(X, \mathsf{Star}_n)$ for arbitrary graphs $X$.  Conway, Elkies, and Martin \cite{conway} investigated a ``projective'' variant of the 15-puzzle problem and its connections to the Mathieu groups.  More recently, Yang \cite{yang} considered a similar problem in this spirit. See also \cite{biniaz, gill1}.

A simple but useful observation is that for any $n$-vertex simple graphs $X$ and $Y$, the graph $\FS(X,Y)$ is isomorphic to $\FS(Y,X)$. Indeed, the map $\sigma\mapsto \sigma^{-1}$ provides the necessary isomorphism. For example, Section~\ref{Sec:Paths} concerns the graphs of the form $\FS(\mathsf{Path}_n,Y)$; by taking inverses of permutations, one could easily rewrite the results in that section in terms of graphs of the form $\FS(X,\mathsf{Path}_n)$. 

\medskip

In Section~\ref{Sec:General}, we prove several general properties of the graphs $\FS(X,Y)$. For example, we show how the decomposition of $X$ into connected components translates into a decomposition of $\FS(X,Y)$ into a disjoint union of Cartesian products. We prove that $\FS(X,Y)$ is disconnected whenever $X$ and $Y$ both have cut vertices, and we provide a lower bound for the number of connected components. We also show that $\FS(X,Y)$ is disconnected whenever $X$ and $Y$ are both bipartite.  

In Section~\ref{Sec:Paths}, we investigate $\FS(\mathsf{Path}_n,Y)$ for arbitrary graphs $Y$. This generalizes Stanley's work in \cite{stanley}. We will see that the connected components correspond bijectively to the acyclic orientations of the complement of $Y$ and that the vertices in each connected component are the linear extensions of a poset naturally associated to the corresponding acyclic orientation. This result is closely related to a known result concerning Coxeter elements in Coxeter systems. We will also see that each connected component of $\FS(\mathsf{Path}_n,Y)$ is isomorphic to the Hasse diagram of a convex subset of the weak order on $\mathfrak S_n$. 

The characterization of the connected components of $\FS(\mathsf{Path}_n,Y)$ is also roughly equivalent to a simple yet useful fact about linear extensions of posets, which we state below as Proposition~\ref{PropToggles}. Let $P=([n],\leq_P)$ be a poset with underlying set $[n]$. A \dfn{linear extension} of $P$ is a permutation $\sigma\in\mathfrak S_n$ such that $\sigma^{-1}(a)\leq \sigma^{-1}(b)$ whenever $a\leq_P b$. Let $\mathcal L(P)$ denote the set of linear extensions of $P$. For each $i\in[n-1]$, define the \dfn{toggle operator} $t_i\colon \mathcal L(P)\to\mathcal L(P)$ by \[t_i(\sigma)=\begin{cases} \sigma\circ(i\,\, i+1), & \mbox{if } \sigma(i)\not\leq_P \sigma(i+1); \\ \sigma, & \mbox{if } \sigma(i)\leq_P \sigma(i+1). \end{cases}\] It is important to note that for each $i\in[n]$ and $\sigma\in\mathcal L(P)$, the permutation $t_i(\sigma)$ is in $\mathcal L(P)$. Furthermore, the map $t_i\colon \mathcal L(P)\to\mathcal L(P)$ is an involution. 

\begin{proposition}\label{PropToggles}
Let $P=([n],\leq_P)$ be a poset. If $\sigma,\sigma'\in\mathcal L(P)$, then there exists a sequence $i_1,\ldots,i_\ell$ of elements of $[n-1]$ such that $\sigma'=(t_{i_\ell}\circ\cdots\circ t_{i_1})(\sigma)$. 
\end{proposition}

Proposition~\ref{PropToggles} is not too difficult to prove by induction on $n$, but it is of fundamental importance in many different places in combinatorics. It is usually used to show that a definition involving a choice of a linear extension is well-defined in the sense that it does not depend on the choice of the linear extension. For example, it has been employed in order to understand Coxeter elements in Coxeter systems \cite[Theorem 1.15]{Develin}, Jucys-Murphy elements and Gelfand-Tsetlin bases \cite[Proposition 8.18]{Meliot}, $\oplus$-diagrams and \reflectbox{L}-diagrams \cite[Proposition 2.6]{Lam}, and the hook-length formula \cite[Lemma 3]{Pak}.

In Section~\ref{Sec:Cycles}, we completely describe the connected components of $\FS(\mathsf{Cycle}_n,Y)$, where $\mathsf{Cycle}_n$ is the cycle graph with vertex set $[n]$ and edge set  $\{\{i,i+1\}:1\leq i\leq n-1\}\cup\{\{n,1\}\}$.
This description is much more involved than (yet very similar in flavor to) the description of the connected components of $\FS(\mathsf{Path}_n,Y)$; it makes use of toric acyclic orientations, which have appeared in many contexts and were formalized in \cite{Develin}. The connected components of $\FS(\mathsf{Cycle}_n,Y)$ can be understood via a new equivalence relation on acyclic orientations of the complement of $Y$, which we call \emph{double-flip equivalence}. This new notion could be of independent interest; it turns out that our analysis of the graphs $\FS(\mathsf{Cycle}_n,Y)$ not only requires an understanding of double-flip equivalence classes but also reciprocally yields interesting structural information about the double-flip equivalence classes. We will see that each toric acyclic orientation of the complement of $Y$ corresponds to $\nu$ isomorphic connected components of $\FS(\mathsf{Cycle}_n,Y)$, where $\nu$ is the greatest common divisor of the sizes of the connected components of the complement of $Y$. One corollary is that $\FS(\mathsf{Cycle}_n,Y)$ is connected if and only if the complement of $Y$ is a forest whose trees have (not necessarily pairwise) coprime sizes. Since Proposition~\ref{PropToggles} is equivalent to our characterization of the connected components of $\FS(\mathsf{Path}_n,Y)$, one can view our characterization of the connected components of $\FS(\mathsf{Cycle}_n,Y)$ as providing a toric analogue of Proposition~\ref{PropToggles}. 

Section~\ref{Sec:Hereditary} concerns a sufficient condition for $\FS(X,Y)$ to be connected. We phrase this result in terms of hereditary classes and Hamiltonian paths. More precisely, we define a \emph{prolongation} of a graph $X$ to be a graph $\widetilde X$ with a Hamiltonian path that itself contains a Hamiltonian path of a subgraph $X^\#$ of $\widetilde X$ such that $X^\#$ is isomorphic to $X$. We will prove that if $X$ has a Hamiltonian path and $\FS(X,Y)$ is connected whenever $Y$ belongs to a hereditary class $\mathcal H$, then $\FS(\widetilde X,\widetilde Y)$ is connected whenever $\widetilde X$ is a prolongation of $X$ and $\widetilde Y$ is in $\mathcal H$. As a corollary, we produce several infinite classes of pairs $(X,Y)$ such that $\FS(X,Y)$ is connected.

In Section~\ref{Sec:Lollipops}, we turn our attention to necessary conditions for $\FS(X,Y)$ to be connected.  We will prove that if $X$ has an induced subgraph isomorphic to $\mathsf{Path}_d$ such that the internal vertices of the path all have degree $2$ in $X$ and $\FS(X,Y)$ is connected, then the minimum degree of $Y$ is at least $d+1$.  By combining this with the main result from Section~\ref{Sec:Hereditary}, we will see that $\FS(\mathsf{Lollipop}_{n-3,3},Y)$ is connected if and only if $Y$ has minimum degree at least $n-2$, where $\mathsf{Lollipop}_{n-3,3}$ is a lollipop graph (which is obtained by identifying an endpoint of a path on $n-2$ vertices with a vertex in the triangle $K_3$). We will also see that $\FS(\mathsf{Lollipop}_{n-3,3},Y)$ is connected if and only if $\FS(\mathsf{D}_n,Y)$ is connected, where $\mathsf D_n$ is the Dynkin diagram of type $D_n$ (which is obtained from $\mathsf{Path}_{n-1}$ by adding the vertex $n$ and the edge $\{n-2,n\}$).  

We end the paper with numerous suggestions for future work in Section~\ref{Sec:Conclusion}. In particular, we define a generalization of friends-and-strangers graphs in which symmetric groups are replaced by arbitrary groups; this construction is given by the intersection of right and left Cayley graphs of the group. Let us also remark that it is very natural to ask probabilistic and extremal questions concerning friends-and-strangers graphs; along with Noga Alon, we address such questions in the article \cite{Typical}.

\subsection{Notation and terminology}\label{Subsec:Notation}
Throughout this article, we assume all graphs are simple. Let $V(G)$ and $E(G)$ denote the vertex set and edge set of a graph $G$. Some specific graphs with vertex set $[n]$ that will play a large role for us are:
\begin{itemize}
\item the \dfn{complete graph} $K_n$, which has edge set $E(K_n)=\{\{i,j\}:1\leq i<j\leq n\}$;
\item the \dfn{star graph} $\mathsf{Star}_n$, which has edge set $E(\mathsf{Star}_n)=\{\{i,n\}:i\in[n-1]\}$;
\item the \dfn{path graph} $\mathsf{Path}_n$, which has edge set $E(\mathsf{Path}_n)=\{\{i,i+1\}:i\in[n-1]\}$;
\item the \dfn{cycle graph} $\mathsf{Cycle}_n$, which has edge set $E(\mathsf{Cycle}_n)=\{\{i,i+1\}:i\in[n-1]\}\cup\{\{n,1\}\}$.
\end{itemize}

The \dfn{disjoint union} of two graphs $G_1,G_2$, denoted $G_1\oplus G_2$, is the graph whose vertex set is the disjoint union $V(G_1)\sqcup V(G_2)$ and whose edge set is the disjoint union $E(G_1)\sqcup E(G_2)$. This definition readily extends to the disjoint union of a family of graphs $G_i$ for $i$ in an index set $I$; we denote the resulting disjoint union by $\bigoplus_{i\in I}G_i$. The \dfn{Cartesian product} of the graphs $G_1, \ldots, G_r$, written $G_1 \square \cdots \square \,G_r$, has vertex set given by the set-theoretic Cartesian product $V(G_1) \times \cdots\times V(G_r)$, where the vertices $(a_1, \ldots, a_r)$ and $(b_1, \ldots, b_r)$ are adjacent if and only if there is an index $i\in[r]$ such that $\{a_i,b_i\} \in E(G_i)$ and $a_j=b_j$ for all $j\in[r]\setminus\{i\}$.

Some additional terminology concerning a graph $G$ is as follows: 
\begin{itemize}
\item The \dfn{complement} of $G$, denoted $\overline G$, is the graph with vertex set $V(\overline G)=V(G)$ such that for all $a,b\in V(G)$ with $a\neq b$, we have $\{a,b\}\in E(\overline G)$ if and only if $\{a,b\}\not\in E(G)$. 
\item An \dfn{isomorphism} from $G$ to a graph $G'$ is a bijection $\varphi\colon V(G)\to V(G')$ such that $\{a,b\}\in E(G)$ if and only if $\{\varphi(a),\varphi(b)\}\in E(G')$. If such an isomorphism exists, we say $G$ and $G'$ are \dfn{isomorphic}, denoted $G\cong G'$. An \dfn{automorphism} of $G$ is an isomorphism from $G$ to itself.   
\item We say a graph $H$ is a \dfn{subgraph} of $G$ if $V(H)\subseteq V(G)$ and $E(H)\subseteq E(G)$; we say the subgraph $H$ is \dfn{induced} if $E(H)=\{\{a,b\}\in E(G):a,b\in V(H)\}$. Given a subset $V_0$ of $V(G)$, let $G\vert_{V_0}$ denote the induced subgraph of $G$ with vertex set $V_0$. 
\item We say $G$ is \dfn{bipartite} if there exists a partition of $V(G)$ into two nonempty sets $A$ and $B$ such that every edge in $E(G)$ has one endpoint in $A$ and one endpoint in $B$; in this case, the pair $\{A,B\}$ is called a \dfn{bipartition} of $G$. 
\item We say $G$ is \dfn{connected} if for all $a,b\in V(G)$, there is a path in $G$ connecting $a$ to $b$; a \dfn{connected component} of $G$ is a maximal connected subgraph of $G$. The \dfn{size} of a connected component $H$ of $G$ is $|V(H)|$. Notice that if $H_1,\ldots,H_r$ are the connected components of $G$, then $G=\bigoplus_{i=1}^r H_i$. 
\item A \dfn{cut vertex} of $G$ is a vertex $v\in V(G)$ such that the graph $G\vert_{V(G)\setminus\{v\}}$ obtained by deleting $v$ has more connected components than $G$. We say $G$ is \dfn{separable} if it is disconnected or has a cut vertex. We say $G$ is \dfn{biconnected} if it has at least $2$ vertices and is not separable. 
\item If $G$ is a graph on $n$ vertices, then a \dfn{Hamiltonian path} in $G$ is a subgraph of $G$ isomorphic to $\mathsf{Path}_n$.
\end{itemize}

\section{General Properties of the Graphs $\FS(X,Y)$}\label{Sec:General} 

To initiate the investigation of the graphs $\FS(X,Y)$, we list some general properties that hold when $X$ and $Y$ are taken from broad classes of graphs. Some of these properties are nicely illustrated by Wilson's results for the graphs of the form $\FS(\mathsf{Star}_n,Y)$, so we will recall those results in this section as well. Several of the general results we prove here will be useful when we consider more specific families of graphs in subsequent sections. 

\begin{proposition}\label{Prop1}
Let $X,\widetilde{X},Y,\widetilde{Y}$ be graphs on $n$ vertices. If $X$ is isomorphic to a subgraph of $\widetilde{X}$ and $Y$ is isomorphic to a subgraph of $\widetilde{Y}$, then $\FS(X,Y)$ is isomorphic to a subgraph of $\FS(\widetilde{X},\widetilde{Y})$. 
\end{proposition}

Taking $\widetilde{X}=\widetilde{Y}=K_n$ in this proposition shows that $\FS(X,Y)$ is isomorphic to a subgraph of $\FS(K_n,K_n)$, the Cayley graph of $\mathfrak S_n$ generated by the set of all transpositions in $\mathfrak S_n$.  

\begin{proof}
We may assume that $V(X)=V(\widetilde{X})=V(Y)=V(\widetilde{Y})=[n]$ and that $X$ and $Y$ are subgraphs of $\widetilde{X}$ and $\widetilde{Y}$, respectively. Let $\{\sigma,\sigma'\}$ be an edge in $\FS(X,Y)$. This means there exists an edge $\{a,b\}$ in $X$ such that $\{\sigma(a),\sigma(b)\}$ is an edge in $Y$, $\sigma(a)=\sigma'(b)$, $\sigma(b)=\sigma'(a)$, and $\sigma(c)=\sigma'(c)$ for all $c\in [n]\setminus\{a,b\}$. Since $\{a,b\}$ is also an edge in $\widetilde{X}$ and $\{\sigma(a),\sigma(b)\}$ is also an edge in $\widetilde{Y}$, it follows that $\{\sigma,\sigma'\}$ is also an edge in $\FS(\widetilde{X},\widetilde{Y})$.  
\end{proof}

For $\sigma\in\mathfrak S_n$, we let $\sgn(\sigma)$ denote the \dfn{sign} of $\sigma$, which is $1$ if $\sigma$ is an even permutation (i.e., a product of an even number of transpositions) and $-1$ if $\sigma$ is an odd permutation (i.e., a product of an odd number of transpositions). 

\begin{proposition}\label{Prop2}
If $X$ and $Y$ are graphs on $n$ vertices, then $\FS(X,Y)$ is bipartite. 
\end{proposition}

\begin{proof}
We may assume $V(X)=V(Y)=[n]$. Let $A_n=\{\sigma\in\mathfrak S_n:\sgn(\sigma)=1\}$ denote the alternating group on $n$ letters. The pair $\{A_n,\mathfrak S_n\setminus A_n\}$ is a bipartition of $\FS(X,Y)$.  
\end{proof}

\begin{proposition}\label{Prop3}
Let $X$ and $Y$ be graphs with $V(X)=V(Y)=[n]$. If $\varphi\colon [n]\to [n]$ is an automorphism of $X$, then the map $\varphi^*\colon \mathfrak S_n\to\mathfrak S_n$ given by $\varphi^*(\sigma)=\sigma\circ\varphi$ is an automorphism of $\FS(X,Y)$.  
\end{proposition}

\begin{proof}
Let $\sigma,\sigma'\in\mathfrak S_n$. The pair $\{\sigma,\sigma'\}$ is an edge in $\FS(X,Y)$ if and only if there exist $i,j\in[n]$ such that $\sigma'=\sigma\circ (i\, j)$, $\{i,j\}\in E(X)$, and $\{\sigma(i),\sigma(j)\}\in E(Y)$. The condition $\sigma'=\sigma\circ (i\, j)$ is equivalent to the condition $\varphi^*(\sigma')=\varphi^*(\sigma)\circ(\varphi^{-1}(i)\,\,\,\varphi^{-1}(j))$. Because $\varphi^{-1}$ is an automorphism of $X$, the condition $\{i,j\}\in E(X)$ is equivalent to the condition that $\{\varphi^{-1}(i),\varphi^{-1}(j)\}\in E(X)$. We have $\{\sigma(i),\sigma(j)\}\in E(Y)$ if and only if $\{\varphi^*(\sigma)(\varphi^{-1}(i)),\varphi^*(\sigma)(\varphi^{-1}(j))\}\in E(Y)$ since these two pairs are actually equal. Thus, $\{\sigma,\sigma'\}\in E(\FS(X,Y))$ if and only if $\{\varphi^*(\sigma),\varphi^*(\sigma')\}\in E(\FS(X,Y))$. 
\end{proof}

The following simple proposition shows that in order to understand the connected components of graphs of the form $\FS(X,Y)$, it suffices to understand what happens when $X$ and $Y$ are connected.  Recall from Section~\ref{Subsec:Notation} the definition of the graph Cartesian product $G_1\square\cdots\square\,G_r$. 

\begin{proposition}\label{Prop4}
Let $X$ and $Y$ be graphs on $n$ vertices, and let $X_1, \ldots, X_r$ be the connected components of $X$. For $i\in[r]$, let $n_i=|V(X_i)|$.  Let $\mathcal{OP}_{n_1,\ldots,n_r}(Y)$ denote the collection of ordered set partitions $(V_1,\ldots,V_r)$ of $V(Y)$ such that $|V_i|=n_i$ for all $i\in[r]$. Then \[\FS(X,Y)\cong\bigoplus_{(V_1,\ldots,V_r)\in\mathcal {OP}_{n_1,\ldots,n_r}(Y)}\left(\FS(X_1,Y\vert_{V_1})\,\square \cdots \square \FS(X_r, Y\vert_{V_r})\right).\] 
\end{proposition}

\begin{proof}
For each bijection $\sigma\colon X\to Y$ and index $i\in[r]$, let $V_i^\sigma=\sigma(X_i)$. Then the ordered set partiton $\rho_\sigma=(V_1^\sigma,\ldots,V_r^\sigma)$ is in $\mathcal{OP}_{n_1,\ldots,n_r}(Y)$. For each $\rho\in\mathcal{OP}_{n_1,\ldots,n_r}(Y)$, let $G_\rho$ be the induced subgraph of $\FS(X,Y)$ whose vertex set consists of the bijections $\sigma\colon X\to Y$ such that $\rho_\sigma=\rho$. If $\sigma'$ is obtained from $\sigma$ by performing an $(X,Y)$-friendly swap (meaning $\{\sigma,\sigma'\}\in\FS(X,Y)$), then we must have $\rho_\sigma=\rho_{\sigma'}$. It follows that every connected component of $\FS(X,Y)$ is contained in one of the induced subgraphs $G_\rho$, so $\displaystyle \FS(X,Y)\cong\bigoplus_{\rho\in\mathcal {OP}_{n_1,\ldots,n_r}(Y)}G_\rho$.

Now fix $\rho=(V_1,\ldots,V_r)\in\mathcal{OP}_{n_1,\ldots,n_r}(Y)$. For $\sigma\in V(G_\rho)$, let $\psi(\sigma)=(\tau_1,\ldots,\tau_r)$, where $\tau_k=\sigma\vert_{X_k}\colon X_k\to Y\vert_{V_k}$. The map $\psi$ is a bijection from $V(G_\rho)$ to $V(\FS(X_1,Y\vert_{V_1}))\times\cdots\times V(\FS(X_r,Y\vert_{V_r}))$. We want to show that $\psi$ is an isomorphism from $G_\rho$ to $\FS(X_1,Y\vert_{V_1})\,\square \cdots \square\FS(X_r, Y\vert_{V_r})$. To do this, choose some edge $\{\sigma,\sigma'\}$ of $G_\rho$. We know that $\sigma'$ is obtained from $\sigma$ by performing an $(X,Y)$-friendly swap across $\{i,j\}$ for some edge $\{i,j\}\in E(X)$. Let $X_m$ be the connected component of $X$ containing the edge $\{i,j\}$. Then $\psi(\sigma')$ is obtained from $\psi(\sigma)$ by performing an $(X_m,Y_m)$-friendly swap across $\{i,j\}$ to the $m$-th coordinate of $\psi(\sigma)$ and leaving all other coordinates of $\psi(\sigma)$ unchanged. This means that $\{\psi(\sigma),\psi(\sigma')\}$ is an edge in $\FS(X_1,Y\vert_{V_1})\, \square \cdots \square \FS(X_r, Y\vert_{V_r})$. This argument is reversible, so $\{\psi(\sigma),\psi(\sigma')\}$ is an edge in $\FS(X_1,Y\vert_{V_1})\,\square \cdots \square \FS(X_r, Y\vert_{V_r})$ if and only if $\{\sigma,\sigma'\}$ in an edge in $G_\rho$.
\end{proof}

Let us now state Wilson's main result about the connected components of $\FS(\mathsf{Star}_n,Y)$. Note that Wilson considers only the case in which $Y$ is biconnected (i.e., $Y$ is connected and does not have a cut vertex). He also excludes cycle graphs, although these graphs are easy to handle separately. Indeed, it is not difficult to show that the graph $\FS(\mathsf{Star}_n,\mathsf{Cycle}_n)$ has $(n-2)!$ connected components, each of which contains exactly $n(n-1)$ vertices.  The connected components of $\FS(\mathsf{Star}_n,\mathsf{Cycle}_n)$ correspond to cyclic orderings of $[n-1]$: indeed, by repeatedly swapping the vertex labeled $n$ with its clockwise neighbor, one can obtain any other labeling of $\mathsf{Cycle}_n$ that witnesses the same cyclic ordering of $[n-1]$. There are $(n-2)!$ cyclic orderings of $[n-1]$; once one is chosen, a vertex $\sigma$ in the corresponding connected component is uniquely determined by specifying $\sigma(n)$ and $\sigma(n-1)$. One other specific graph that Wilson must consider separately is the graph \[\theta_0=\begin{array}{l}\includegraphics[height=1.55cm]{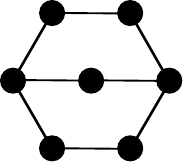}\end{array}.\]

\begin{theorem}[\!\!\cite{wilson}]\label{ThmWilson}
Let $Y$ be a biconnected graph on $n\geq 3$ vertices that is not isomorphic to $\theta_0$ or $\mathsf{Cycle}_n$. If $Y$ is not bipartite, then $\FS(\mathsf{Star}_n,Y)$ is connected. If $Y$ is bipartite, then $\FS(\mathsf{Star}_n,Y)$ has exactly $2$ connected components, each of size $n!/2$. The graph $\FS(\mathsf{Star}_7,\theta_0)$ has exactly $6$ connected components.   
\end{theorem}

We now show that when both $X$ and $Y$ are separable, the resulting $\FS(X,Y)$ has a very large number of components; it is for this reason that Wilson~\cite{wilson} restricted his investigation of $\FS(\mathsf{Star}_n,Y)$ to the case where $Y$ is biconnected.  By Proposition~\ref{Prop4}, it suffices to consider the case where $X$ and $Y$ are connected but separable.

\begin{proposition}\label{prop5}
Let $X$ and $Y$ be connected graphs, each on $n \geq 3$ vertices. Suppose $x_0 \in X$ and $y_0 \in Y$ are cut vertices such that the connected components of $X\vert_{V(X)\setminus\{x_0\}}$ are $X_1, \ldots, X_r$ and the connected components of $Y\vert_{V(Y)\setminus\{y_0\}}$ are $Y_1, \ldots, Y_s$. Let $\mathcal M=\mathcal{M}(X,Y,x_0,y_0)$ denote the set of $r \times s$ matrices with nonnegative integer entries in which the $i$-th row sums to $|V(X_i)|$ and the $j$-th column sums to $|V(Y_j)|$.  Then the number of connected components of $\FS(X,Y)$ is at least $|\mathcal{M}|$.
\end{proposition}

For more information about $\mathcal{M}$ (including size estimates), we direct the reader to \cite{barvinok,diaconis} and the references therein. Note that we always have $|\mathcal M|\geq 2$ (so that $\FS(X,Y)$ is disconnected) since $r, s \geq 2$ by the definition of a cut vertex.  We also remark that, as will be apparent in the proof, this lower bound on the number of connected components of $\FS(X,Y)$ is an equality when the graphs $X\vert_{V(X_i) \cup \{x_0\}}$ and $Y\vert_{V(Y_j) \cup \{y_0\}}$ are all complete graphs.

\begin{proof}
Let $\Sigma_0=\{\sigma \in V(\FS(X,Y)): \sigma(x_0)=y_0\}$.  For each $\sigma \in \Sigma_0$, define the $r \times s$ \emph{incidence matrix} $M(\sigma)=(m_{i,j}(\sigma)) \in \mathcal{M}$ via $m_{i,j}(\sigma)=|\sigma(V(X_i)) \cap V(Y_j)|$. It is immediate that every $M \in \mathcal{M}$ arises as $M(\sigma)$ for some $\sigma$. Therefore, it suffices to show that $M(\sigma)=M(\sigma')$ whenever $\sigma,\sigma'\in\Sigma_0$ are in the same connected component of $\FS(X,Y)$. Fix some $\sigma \in \Sigma_0$. We will actually prove the following stronger assertion for every $\tau \in V(\FS(X,Y))$ in the same connected component as $\sigma$:
\begin{itemize}
    \item if $\tau\in V(\Sigma_0)$, then $M(\tau)=M(\sigma)$;
    \item if $X_{i_0}$ is the connected component of $X$ containing $\tau^{-1}(y_0)$, then $|\tau(V(X_{i_0}) \cup \{x_0\}) \cap V(Y_j)|=m_{i_0,j}(\sigma)$ and $|\tau(V(X_i)) \cap V(Y_j)|=m_{i,j}(\sigma)$ for every $i \neq i_0$.
\end{itemize}
Let us assume that these conditions hold for some $\tau$ in the same connected component as $\sigma$. We will show that the claim also holds for every vertex $\tau'$ that is adjacent to $\tau$ in $\FS(X,Y)$, which will complete the proof. Let $\tau'$ be obtained from $\tau$ by the application of an $(X,Y)$-friendly swap across the edge $\{x_1,x_2\}\in E(X)$. 

First, suppose that $\tau\in V(\Sigma_0)$ (i.e., $\tau(x_0)=y_0$).  If $x_0\not\in\{x_1,x_2\}$, then $\{x_1,x_2\}$ lies within a single connected component $X_i$ of $X$, so $\tau' \in V(\Sigma_0)$ and $M(\tau')=M(\tau)$.  It is also straightforward to check that the claim holds if instead (without loss of generality) $x_1=x_0$ and $x_2 \in V(X_{i_0})$.

Second, suppose that $\tau^{-1}(y_0) \in V(X_{i_0})$.  Again, the result is straightforward if $x_0 \notin \{x_1,x_2\}$ or if $y_0 \notin \{\tau(x_1), \tau(x_2)\}$.  The only remaining case is (without loss of generality) where $x_0=x_1$ and $\tau(x_2)=y_0$ (where we know that $x_2 \in X_{i_0}$).  Here, the desired condition is again easy to check.
\end{proof}

Part of Wilson's Theorem~\ref{ThmWilson} says that $\FS(\mathsf{Star}_n,Y)$ is disconnected whenever $Y$ is bipartite. The following result generalizes this to all bipartite $X$ and $Y$. As before, let $\sgn(\sigma)$ be $1$ if $\sigma$ is an even permutation and $-1$ if $\sigma$ is an odd permutation.  

\begin{proposition}\label{prop6}
If $X$ and $Y$ are bipartite graphs, each on $n \geq 3$ vertices, then $\FS(X,Y)$ is disconnected.
\end{proposition}

\begin{proof}
We may assume $V(X)=V(Y)=[n]$. Let $\{A_X,B_X\}$ be a bipartition of $X$, and let $\{A_Y,B_Y\}$ be a bipartition of $Y$. For each $\sigma\in\mathfrak S_n$, let \[p(\sigma)=|\sigma(A_X) \cap A_Y|+\frac{\sgn(\sigma)+1}{2}.\]
We claim that if $\sigma$ and $\sigma'$ are adjacent in $\FS(X,Y)$, then $p(\sigma)$ and $p(\sigma')$ have the same parity. Indeed, suppose $\sigma'$ is obtained from $\sigma$ by performing an $(X,Y)$-friendly swap across $\{i,j\}$. This means that $\{i,j\}\in E(X)$ and $\{\sigma(i),\sigma(j)\}\in E(Y)$. The numbers $i$ and $j$ belong to different sets in the bipartition $\{A_X,B_X\}$, and the numbers $\sigma(i)$ and $\sigma(j)$ belong to different sets in the bipartition $\{A_Y,B_Y\}$. Since $\sigma'=\sigma\circ(i\,j)$, it follows that $|\sigma(A_X)\cap A_Y|$ and $|\sigma'(A_X)\cap A_Y|$ differ by exactly $1$ and that $(\sgn(\sigma)+1)/2$ and $(\sgn(\sigma')+1)/2$ differ by exactly $1$. This proves the claim that $p(\sigma)\equiv p(\sigma')\pmod 2$. It follows that $p(\tau)\equiv p(\tau')\pmod 2$ whenever $\tau$ and $\tau'$ are in the same connected component of $\FS(X,Y)$. 

Now fix $\tau\in\mathfrak S_n$. Since $n\geq 3$, there exist elements $a$ and $b$ of $[n]$ that lie in the same set in the bipartition $\{A_X,B_X\}$. Let $\tau'=\tau\circ (a\, b)$. Then $p(\tau)$ and $p(\tau')$ differ by $1$, so $\tau$ and $\tau'$ must lie in different connected components of $\FS(X,Y)$. 
\end{proof}

\begin{remark}\label{Rem1}
Note that $\mathsf{Star}_n$ is bipartite with bipartition $\{[n-1],\{n\}\}$. Suppose that $n\geq 3$ and that $Y$ is a biconnected graph on the vertex set $[n]$ that is not isomorphic to $\mathsf{Cycle}_n$ or $\theta_0$. If $Y$ is bipartite with bipartition $\{A_Y,B_Y\}$, then it follows from Theorem~\ref{ThmWilson} that $\FS(\mathsf{Star}_n,Y)$ has exactly $2$ connected components. Appealing to the proof of Proposition~\ref{prop6}, we find that the connected components of $\FS(\mathsf{Star}_n,Y)$ must be precisely \[\{\sigma\in\mathfrak S_n:p(n)\equiv 0\pmod 2\}\quad\text{and}\quad\{\sigma\in\mathfrak S_n:p(n)\equiv 1\pmod 2\},\] where \[p(\sigma)=|\sigma([n-1])\cap A_Y|+\frac{\sgn(\sigma)+1}{2}.\] Now suppose $X$ is a graph with vertex set $[n]$ such that $\mathsf{Star}_n$ is a proper subgraph of $X$. This means that there exists an edge $\{i,j\}\in E(X)$ with $n \notin \{i,j\}$. There exists $\sigma\in\mathfrak S_n$ such that $\sigma(i)$ and $\sigma(j)$ are adjacent in $Y$. The permutation $\sigma'=\sigma\circ(i\, j)$ is adjacent to $\sigma$ in $\FS(X,Y)$, and $p(\sigma)\not\equiv p(\sigma')\pmod 2$. This means that $\sigma$ and $\sigma'$ are in different connected components of $\FS(\mathsf{Star}_n,Y)$, but are in the same connected component of $\FS(X,Y)$. Since each of the $2$ connected components of $\FS(\mathsf{Star}_n,Y)$ is contained in a connected component of $\FS(X,Y)$, it follows that $\FS(X,Y)$ is connected.  This fact demonstrates a sense in which Wilson's result is sharp.
\end{remark}

\section{Paths}\label{Sec:Paths} 

In this section, we describe the structure of $\FS(\mathsf{Path}_n,Y)$. We will see that this characterization is closely related to a known result about Coxeter elements of Coxeter systems.  

In what follows, we let $G$ be a graph with vertex set $[n]$.  We obtain an \dfn{orientation} of $G$ by choosing a direction for each of its edges.  An orientation is \dfn{acyclic} if it does not contain a directed cycle.  Let $\Acyc(G)$ denote the set of acyclic orientations of $G$.
For every $\alpha\in\Acyc(G)$, we obtain a poset $([n],\leq_\alpha)$ by declaring that $i\leq_\alpha j$ if and only if the directed graph $\alpha$ contains a directed path starting at the vertex $i$ and ending at the vertex $j$. (When $i=j$, we can use the $1$-vertex path with no edges.) We write $\mathcal L(\alpha)$ for the set of linear extensions of $([n],\leq_\alpha)$, as defined in the introduction. For each permutation $\sigma\in\mathfrak S_n$, there is a unique acyclic orientation $\alpha_G(\sigma)\in\Acyc(G)$ such that $\sigma\in\mathcal L(\alpha_G(\sigma))$. Indeed, $\alpha_G(\sigma)$ is obtained by directing each edge $\{i,j\}$ of $G$ from $i$ to $j$ if and only if $\sigma^{-1}(i)<\sigma^{-1}(j)$.

\begin{figure}[ht]
\begin{center}
\includegraphics[height=2.88cm]{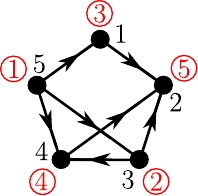}
\caption{A graph $G$ labeled according to the permutation $\sigma=53142$. Each vertex $i$ is labeled with a plain black label $i$ and a circled red label $\sigma^{-1}(i)$. The edges are oriented so as to form the acyclic orientation $\alpha_G(\sigma)$.}
\label{Fig1}
\end{center}  
\end{figure}

\begin{theorem}\label{Thm1}
Let $Y$ be a graph with vertex set $[n]$. For each $\alpha\in\Acyc(\overline Y)$, choose a linear extension $\sigma_\alpha\in\mathcal L(\alpha)$. Let $H_\alpha$ be the connected component of $\FS(\mathsf{Path}_n,Y)$ containing $\sigma_\alpha$. The connected component $H_\alpha$ depends only on $\alpha$ (not on the specific choice of $\sigma_\alpha$), and its vertex set is $\mathcal L(\alpha)$. Moreover, \[\FS(\mathsf{Path}_n,Y)=\bigoplus_{\alpha\in\Acyc(\overline Y)}H_\alpha.\]
\end{theorem}

\begin{proof}
Choose distinct $\sigma,\sigma'\in\mathfrak S_n$, and let $\alpha=\alpha_{\overline Y}(\sigma)$. The permutations $\sigma$ and $\sigma'$ are adjacent in $\FS(\mathsf{Path}_n,Y)$ if and only if there exists $i\in[n-1]$ such that $\{\sigma(i),\sigma(i+1)\}$ is an edge in $Y$ and $\sigma'=\sigma\circ (i\,\,i+1)$. Notice that in the poset $([n],\leq_\alpha)$, the element $\sigma(i+1)$ either covers $\sigma(i)$ or is incomparable to $\sigma(i)$. This means that $\sigma(i)$ and $\sigma(i+1)$ are adjacent in $Y$ (i.e., nonadjacent in $\overline Y$) if and only if they are incomparable in $([n],\leq_\alpha)$. Letting $t_1,\ldots,t_{n-1}$ be the toggle operators from Proposition~\ref{PropToggles} (with respect to the poset $([n],\leq_\alpha)$), we see that $\sigma$ and $\sigma'$ are adjacent if and only if $\sigma'=t_i(\sigma)$ for some $i\in[n-1]$. It follows that if $\sigma$ and $\sigma'$ are adjacent in $\FS(\mathsf{Path}_n,Y)$, then $\alpha_{\overline Y}(\sigma)=\alpha_{\overline Y}(\sigma')$. Using Proposition~\ref{PropToggles}, we see that $\sigma$ and $\sigma'$ are in the same connected component of $\FS(\mathsf{Path}_n,Y)$ if and only if $\alpha_{\overline Y}(\sigma)=\alpha_{\overline Y}(\sigma')$. This is equivalent to the statement of the theorem. 
\end{proof}

It is well known that the number of acyclic orientations of a graph $G$ is equal to the evaluation $T_G(2,0)$ of the Tutte polynomial of $G$ (see \cite{Bollobas} for the definition of the Tutte polynomial).  Furthermore, a graph with at least $1$ edge has at least $2$ acyclic orientations. Hence, we have the following corollary. 

\begin{corollary}\label{Cor1}
Let $Y$ be a graph with vertex set $[n]$. The number of connected components of $\FS(\mathsf{Path}_n,Y)$ is $T_{\overline Y}(2,0)$. In particular, $\FS(\mathsf{Path}_n,Y)$ is connected if and only if $Y=K_n$. 
\end{corollary}

Recall that a \dfn{Coxeter system} is a pair $(W,S)$, where $W$ is a group with generating set $S=\{s_1,\ldots,s_n\}$ and presentation $W=\langle S:(s_is_j)^{m_{i,j}}=1\rangle$. Here, the exponents $m_{i,j}$ are elements of $\{1,2,3,\ldots\}\cup\{\infty\}$ such that $m_{ii}=1$ for all $i\in[n]$ and $m_{i,j}=m_{j,i}\geq 2$ whenever $i\neq j$. Note that the elements $s_i$ and $s_j$ commute if and only if $m_{i,j}\leq 2$. The \dfn{Coxeter graph}\footnote{Coxeter graphs often have edge labels that encode the exponents $m_{i,j}$, but we will ignore those labels here.} associated to the Coxeter system $(W,S)$ is the simple graph with vertex set $S$ in which vertices $s_i$ and $s_j$ are adjacent if and only if $m_{i,j}\geq 3$ (i.e., $s_is_j\neq s_js_i$). A \dfn{Coxeter element} (several authors say \dfn{standard Coxeter element}) of $(W,S)$ is an element of $W$ of the form $s_{\sigma(1)}\cdots s_{\sigma(n)}$, where $\sigma\in\mathfrak S_n$.

Now let $Y$ be a graph with vertex set $[n]$. There exists a Coxeter system $(W,S)$ whose Coxeter graph is $\overline{Y}$, where we identify the vertex $i\in[n]=V(\overline Y)$ with the element $s_i\in S$. With this identification, every permutation $\sigma\in\mathfrak S_n$ gives rise to a word $s_{\sigma(1)}\cdots s_{\sigma(n)}$, which represents a Coxeter element of $(W,S)$. Two such words represent the same element of $W$ if and only if one can be obtained from the other by repeatedly applying the commutation relations $s_is_j=s_js_i$, which hold when $i$ and $j$ are adjacent in $Y$. Applying such a commutation relation to a word $s_{\sigma(1)}\cdots s_{\sigma(n)}$ means that we swap the factors $s_{\sigma(i)}$ and $s_{\sigma(i+1)}$ for some $i\in[n-1]$ such that $\{\sigma(i),\sigma(i+1)\}$ is an edge in $Y$. This corresponds precisely to applying a $(\mathsf{Path}_n,Y)$-friendly swap to the permutation $\sigma$. Hence, the Coxeter elements $s_{\sigma(1)}\cdots s_{\sigma(n)}$ and $s_{\sigma'(1)}\cdots s_{\sigma'(n)}$ are equal if and only if $\sigma$ and $\sigma'$ are in the same connected component of $\FS(\mathsf{Path}_n,Y)$. It follows that Theorem~\ref{Thm1} is equivalent to the following standard theorem about Coxeter elements (see \cite{Cartier, Develin}). 

\begin{theorem}[\!\!\!\cite{Cartier, Develin}]\label{ThmCoxeter} 
Let $(W,S)$ be a Coxeter system with Coxeter graph $G$, and write $S=\{s_1,\ldots,s_n\}$. Identify each vertex $s_i$ of $G$ with the element $i$ of $[n]$. For each acyclic orientation $\alpha\in\Acyc(G)$, choose a linear extension $\sigma_\alpha$ of $([n],\leq_\alpha)$. The Coxeter element $s_{\sigma_\alpha(1)}\cdots s_{\sigma_\alpha(n)}$ depends only on $\alpha$, not on the specific linear extension $\sigma_\alpha$. Furthermore, the map $\alpha\mapsto s_{\sigma_\alpha(1)}\cdots s_{\sigma_\alpha(n)}$ is a bijection from $\Acyc(G)$ to the set of Coxeter elements of $(W,S)$.       
\end{theorem}

It is worth mentioning that Theorem~\ref{ThmCoxeter}, and hence Theorem~\ref{Thm1}, can also be interpreted outside of a Coxeter-theoretic framework
using certain Cartier--Foata traces \cite{Cartier}.

We end this section with a brief description of how the connected components of $\FS(\mathsf{Path}_n,Y)$ inherit their structure from the graph $\FS(\mathsf{Path}_n,K_n)$. It is immediate that every edge in $\FS(\mathsf{Path}_n,Y)$ is also an edge in $\FS(\mathsf{Path}_n,K_n)$. On the other hand, suppose $\sigma,\tau\in\mathfrak S_n$ are in the same connected component of $\FS(\mathsf{Path}_n,Y)$ and that they are adjacent in $\FS(\mathsf{Path}_n,K_n)$. There must exist $i\in[n-1]$ such that $\tau=\sigma\circ(i\,\,i+1)$. Since $\sigma$ and $\tau$ are in the same connected component of $\FS(\mathsf{Path}_n,Y)$, Theorem~\ref{Thm1} tells us that $\alpha_{\overline Y}(\sigma)=\alpha_{\overline Y}(\tau)$. This implies that $\sigma(i)$ and $\sigma(i+1)$ are adjacent in $Y$, so $\tau$ is obtained from $\sigma$ by performing a $(\mathsf{Path}_n,Y)$-friendly swap across $\{i,i+1\}$. Hence, $\sigma$ and $\tau$ are adjacent in $\FS(\mathsf{Path}_n,Y)$. This yields the following corollary to Theorem~\ref{Thm1}. 

\begin{corollary}
Let $Y$ be a graph with vertex set $[n]$. Let $H_\alpha$ be as in Theorem~\ref{Thm1}, so that \[\FS(\mathsf{Path}_n,Y)=\bigoplus_{\alpha\in\Acyc(\overline Y)}H_\alpha.\] Then $H_\alpha=\FS(\mathsf{Path}_n,K_n)\vert_{\mathcal L(\alpha)}$ for every $\alpha\in\Acyc(\overline Y)$. 
\end{corollary}

\begin{remark}
Note that $\FS(\mathsf{Path}_n,K_n)$ is the Hasse diagram of the weak order on $\mathfrak S_n$. A subset $A$ of a poset $P$ is called \dfn{convex} if for all $a,c\in A$ and all $b\in P$ satisfying $a<_Pb<_Pc$, we have $b\in A$. Bj\"orner and Wachs \cite{Bjorner} showed that a subset $A\subseteq \mathfrak S_n$ is a convex subset of the weak order on $\mathfrak S_n$ if and only if $A$ is the set of linear extensions of a poset. It follows from Theorem~\ref{Thm1} that the set of vertices of a connected component of $\FS(\mathsf{Path}_n,Y)$ is a convex subset of the weak order on $\mathfrak S_n$. 
\end{remark}

\section{Cycles}\label{Sec:Cycles}

In this section, we describe the structure of $\FS(\mathsf{Cycle}_n,Y)$. We will see that many of the results in this section are similar in form to those in the previous section, with acyclic orientations and posets replaced by their appropriate toric analogues. One of the key components of the proof of Theorem~\ref{Thm1} was the fundamental Proposition~\ref{PropToggles}, which states that the group of bijections generated by the toggles $t_1,\ldots,t_{n-1}$ acts transitively on the set of linear extensions of an $n$-element poset. Some new ideas will be needed to establish the theorems in this section. 

Let $G$ be a graph with vertex set $[n]$. A \dfn{source} of an acyclic orientation $\alpha$ of $G$ is a vertex of in-degree $0$; a \dfn{sink} of $\alpha$ is a vertex of out-degree $0$. If $v$ is a source or a sink of $\alpha$, then we can obtain a new acyclic orientation of $G$ by reversing the directions of all edges incident to $v$. We call this operation a \dfn{flip}. Two acyclic orientations $\alpha,\alpha'\in\Acyc(G)$ are \dfn{torically equivalent}, denoted $\alpha\sim\alpha'$, if $\alpha'$ can be obtained from $\alpha$ via a sequence of flips. The equivalence classes in $\Acyc(G)/\!\!\sim$ are called \dfn{toric acyclic orientations}. We denote the toric acyclic orientation containing the acyclic orientation $\alpha$ by $[\alpha]_\sim$. 

Toric acyclic orientations have been studied in many different forms \cite{Chen, Eriksson, Macauley, Mosesjan, Pretzel, Speyer}; the article \cite{Develin} formalizes a systematic framework for their investigation. One of the reasons for the use of the word ``toric'' stems from their connection with hyperplane arrangements. Indeed, consider the \dfn{graphical arrangement} of the graph $G$, which is the hyperplane arrangement $\mathcal A(G)$ in $\mathbb R^n$ consisting of the hyperplanes of the form $\{x\in\mathbb R^n:x_i=x_j\}$ for all edges $\{i,j\}$ of $G$. Let $\mathcal A_{\tor}(G)=\pi(\mathcal A(G))$, where $\pi\colon \mathbb R^n\to\mathbb R^n/\mathbb Z^n$ is the natural projection map. The \dfn{toric chambers} of $\mathcal A_{\tor}(G)$ are the connected components of $(\mathbb R^n/\mathbb Z^n)\setminus\mathcal A_{\tor}(G)$. As shown in \cite{Develin}, there is a natural one-to-one correspondence between the toric chambers of $\mathcal A_{\tor}(G)$ and the toric acyclic orientations of $G$. This is analogous to the standard fact that the chambers of $\mathcal A(G)$ (the connected components of $\mathbb R^n\setminus\mathcal A(G)$) are in bijection with $\Acyc(G)$. 

Another (related) motivation for the term ``toric'' comes from observing that flips encode what happens to the acyclic orientation associated to a permutation when we cyclically shift the permutation. To make this more precise, we let $\varphi\colon [n]\to[n]$ be the cyclic permutation given by $\varphi(i)=i+1 \pmod{n}$ and consider the map $\varphi^*\colon \mathfrak S_n\to\mathfrak S_n$ defined by $\varphi^*(\sigma)=\sigma\circ\varphi$. Recall that $\alpha_G(\sigma)$ denotes the unique acyclic orientation of $G$ such that $\sigma\in\mathcal L(\alpha_G(\sigma))$. The vertex $\sigma(1)$ of $G$ is a source of $\alpha_G(\sigma)$. It is not hard to show that the acyclic orientation $\alpha_G(\varphi^*(\sigma))$ is obtained from $\alpha_G(\sigma)$ by flipping the vertex $\sigma(1)$ from a source into a sink. Consequently, the acyclic orientations $\alpha_G((\varphi^*)^k(\sigma))$ for $0\leq k\leq n-1$ are all torically equivalent. It is also helpful to keep in mind that the map $\varphi^*$ has order $n$ and that the acyclic orientation $\alpha_G((\varphi^*)^{-1}(\sigma))=\alpha_G((\varphi^*)^{n-1}(\sigma))$ is obtained from $\alpha_G(\sigma)$ by flipping the vertex $\sigma(n)$ from a sink into a source. 
Note that $\varphi$ is an automorphism of $\mathsf{Cycle}_n$. Proposition~\ref{Prop3} tells us that for every graph $Y$ with vertex set $[n]$, the map $\varphi^*$ is an automorphism of $\FS(\mathsf{Cycle}_n,Y)$. Given an induced subgraph $H$ of $\FS(\mathsf{Cycle}_n,Y)$, we write $\varphi^*(H)$ for the induced subgraph of $\FS(\mathsf{Cycle}_n,Y)$ on the vertex set $\varphi^*(V(H))$. 

We define a \dfn{linear extension}\footnote{Note that our notion of a linear extension of a toric acyclic orientation differs from the definition of a ``toric total order'' in \cite{Develin}. Indeed, that article defines a toric total order of $[\alpha]_\sim$ to be a cyclic equivalence class $\{(\varphi^*)^k(\sigma):0\leq k\leq n-1\}$ such that $\sigma$ is (using our definition) a linear extension of $[\alpha]_{\sim}$.} of $[\alpha]_{\sim}$ to be a permutation $\sigma$ such that there exists an acyclic orientation $\widehat\alpha\in[\alpha]_{\sim}$ with $\sigma\in\mathcal L(\widehat\alpha)$. Letting $\mathcal L([\alpha]_\sim)$ denote the set of linear extensions of $[\alpha]_{\sim}$, we have \[\mathcal L([\alpha]_{\sim})=\bigcup_{\widehat\alpha\in[\alpha]_{\sim}}\mathcal L(\widehat\alpha).\] For every permutation $\sigma$, the unique toric acyclic orientation of $G$ that has $\sigma$ as a linear extension is $[\alpha_G(\sigma)]_{\sim}$.    

We will also need a new equivalence relation on acyclic orientations of a graph.  Suppose $\alpha\in\Acyc(G)$ has a source $u$ and a sink $v$ such that $u$ and $v$ are not adjacent. We can simultaneously flip $u$ into a sink and flip $v$ into a source; we call this a \dfn{double flip}. We say two acyclic orientations $\alpha,\alpha'\in\Acyc(G)$ are \dfn{double-flip equivalent}, denoted $\alpha\approx\alpha'$, if $\alpha'$ can be obtained from $\alpha$ via a sequence of double flips.  Let $[\alpha]_{\approx}$ denote the equivalence class in $\Acyc(G)/\!\!\approx$ that contains $\alpha$.  Note that every equivalence class in $\Acyc/\!\!\sim$ is a union of equivalence classes in $\Acyc/\!\!\approx$ because each double flip can be viewed as two individual flips performed one after the other; here, it is crucial that we require $u$ and $v$ to non-adjacent in the definition of a double flip.  A \dfn{linear extension} of a double-flip equivalence class $[\alpha]_{\approx}$ is a permutation $\sigma$ such that $\sigma\in\mathcal L(\widehat\alpha)$ for some $\widehat\alpha\in[\alpha]_{\approx}$. Letting $\mathcal L([\alpha]_{\approx})$ denote the set of linear extensions of $[\alpha]_{\approx}$, we have \[\mathcal L([\alpha]_{\approx})=\bigcup_{\widehat\alpha\in[\alpha]_{\approx}}\mathcal L(\widehat\alpha).\]

\begin{figure}[ht]
\begin{center}
\includegraphics[width=.8\linewidth]{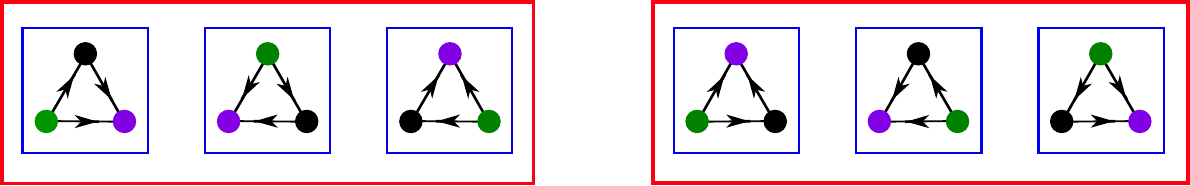}
\caption{Each red box encompasses a toric acyclic orientation of $K_3$. Each blue box encompasses a double-flip equivalence class of $K_3$. Note that each double-flip equivalence class contains a single acyclic orientation; it is impossible to perform a double flip on an acyclic orientation of a complete graph because sources and sinks are all adjacent. We have colored sinks purple and sources green.}
\label{Fig2}
\end{center}  
\end{figure}

\begin{figure}[ht]
\begin{center}
\includegraphics[width=\linewidth]{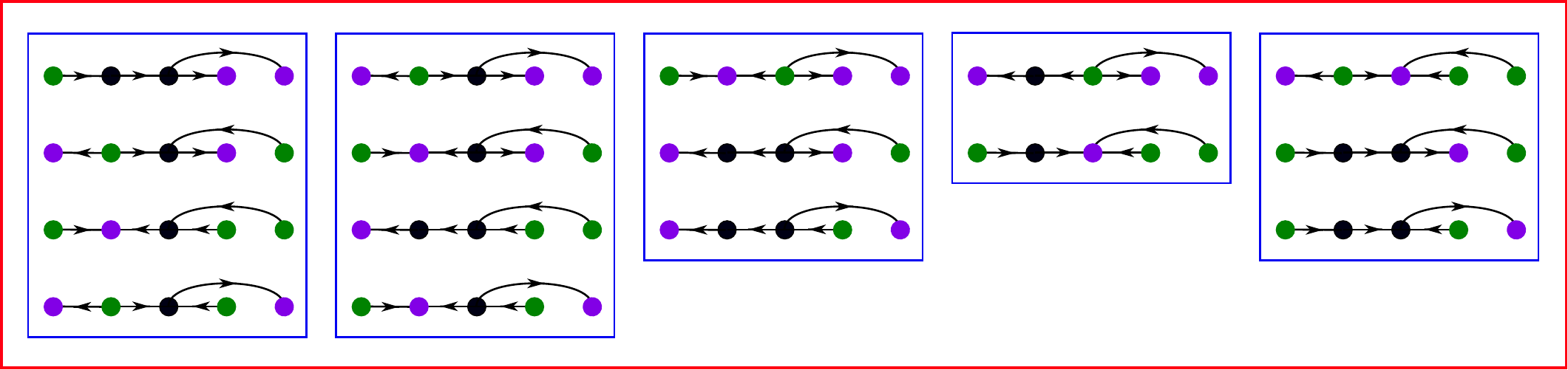}
\caption{This graph has only $1$ toric acyclic orientation, which is the union of the $5$ double-flip equivalence classes encompassed in the blue boxes. We have colored sinks purple and sources green.}
\label{Fig3}
\end{center}  
\end{figure}

The next theorem tells us that the connected components of $\FS(\mathsf{Cycle}_n,Y)$ are the sets of linear extensions of double-flip equivalence classes of the acyclic orientations of $\overline Y$. The basic idea behind this result is as follows. We already know by Theorem~\ref{Thm1} that the connected components of $\FS(\mathsf{Path}_n,Y)$ are the sets of linear extensions of the acyclic orientations of $\overline Y$. Performing a $(\mathsf{Cycle}_n,Y)$-friendly swap to a bijection $\sigma$ across the edge $\{n,1\}$ corresponds precisely to changing the source $\sigma(1)$ into a sink and changing the sink $\sigma(n)$ into a source. When we perform this friendly swap, the vertices $\sigma(1)$ and $\sigma(n)$ must be adjacent in $Y$, meaning they are not adjacent in $\overline Y$. This is precisely why we forbid $u$ and $v$ from being adjacent in the definition of a double flip. One could also view each double-flip equivalence class as a cyclic version of a Cartier--Foata trace \cite{Cartier}, where commutations are additionally allowed to wrap around the end of the word.

\begin{theorem}\label{Thm2}
Let $Y$ be a graph with vertex set $[n]$. For each $[\alpha]_{\approx}\in\Acyc(\overline Y)/{\approx}$, choose a linear extension $\sigma_{[\alpha]_{\approx}}\in\mathcal L([\alpha]_{\approx})$. Let $H_{[\alpha]_{\approx}}$ be the connected component of $\FS(\mathsf{Cycle}_n,Y)$ containing $\sigma_{[\alpha]_{\approx}}$.  The connected component $H_{[\alpha]_{\approx}}$ depends only on $[\alpha]_{\approx}$ (not on the specific choice of $\sigma_{[\alpha]_{\approx}}$), and its vertex set is $\mathcal L([\alpha]_{\approx})$. Moreover, \[\FS(\mathsf{Cycle}_n,Y)=\bigoplus_{[\alpha]_{\approx}\in\Acyc(\overline Y)/\approx}H_{[\alpha]_{\approx}}.\]
\end{theorem}

\begin{proof}
The theorem is equivalent to the statement that two permutations $\sigma,\sigma'\in\mathfrak S_n$ are in the same connected component of $\FS(\mathsf{Cycle}_n,Y)$ if and only if $\alpha_{\overline Y}(\sigma)\approx\alpha_{\overline Y}(\sigma')$. First, suppose $\sigma$ and $\sigma'$ are adjacent in $\FS(\mathsf{Cycle}_n,Y)$. This means that $\sigma'$ is obtained from $\sigma$ by performing a $(\mathsf{Cycle}_n,Y)$-friendly swap across an edge $\{i,j\}\in E(\mathsf{Cycle}_n)$. If $i\leq n-1$ and $j=i+1$, then this $(\mathsf{Cycle}_n,Y)$-friendly swap is also a $(\mathsf{Path}_n,Y)$-friendly swap, so $\sigma$ and $\sigma'$ are adjacent in $\FS(\mathsf{Path}_n,Y)$. By Theorem~\ref{Thm1}, this implies that $\alpha_{\overline Y}(\sigma)=\alpha_{\overline Y}(\sigma')$. Now suppose $\{i,j\}=\{n,1\}$. The vertices $\sigma(1)$ and $\sigma(n)$ are adjacent in $Y$, so they are not adjacent in $\overline Y$. The former vertex is a source in $\alpha_{\overline Y}(\sigma)$, and the latter is a sink in $\alpha_{\overline Y}(\sigma)$. It is straightforward to see that $\alpha_{\overline Y}(\sigma')$ is obtained from $\alpha_{\overline Y}(\sigma)$ by performing a double flip that simultaneously flips $\sigma(1)$ from a source to a sink and flips $\sigma(n)$ from a sink to a source. Consequently, $\alpha_{\overline Y}(\sigma)\approx\alpha_{\overline Y}(\sigma')$. This shows that $\alpha_{\overline Y}(\sigma)\approx\alpha_{\overline Y}(\sigma')$ whenever $\sigma$ and $\sigma'$ are adjacent in $\FS(\mathsf{Cycle}_n,Y)$, which implies that $\alpha_{\overline Y}(\sigma)\approx\alpha_{\overline Y}(\sigma')$ whenever $\sigma$ and $\sigma'$ are in the same connected component of $\FS(\mathsf{Cycle}_n,Y)$. 

To prove the converse, notice that the vertex set of a connected component of $\FS(\mathsf{Cycle}_n,Y)$ is a union of vertex sets of connected components of $\FS(\mathsf{Path}_n,Y)$. By Theorem~\ref{Thm1}, the vertex set of each connected component of $\FS(\mathsf{Path}_n,Y)$ is the set of linear extensions of some acyclic orientation of
$\overline Y$. Therefore, we need only show that if $\alpha,\alpha'\in\Acyc(\overline Y)$ satisfy $\alpha\approx\alpha'$, then there exist $\tau\in\mathcal L(\alpha)$ and $\tau'\in\mathcal L(\alpha')$ such that $\tau$ and $\tau'$ are in the same connected component of $\FS(\mathsf{Cycle}_n,Y)$. To prove this, it suffices to consider the case in which $\alpha'$ is obtained from $\alpha$ by performing a double
flip. Thus, let us assume that there exist a source $u$ of $\alpha$ and a sink $v$ of $\alpha$ (with $\{u,v\}\not\in E(\overline Y)$) such that $\alpha'$ is obtained from $\alpha$ by flipping $u$ into a sink and flipping $v$ into a source. It is straightforward to see that there is a linear extension $\tau$ of $\alpha$ such that $\tau(1)=u$ and $\tau(n)=v$. The vertices $u$ and $v$ are adjacent in $Y$ (since they are not adjacent in $\overline Y$), so the operation that changes $\tau$ into $\tau':=\tau\circ(1\,n)$ is a $(\mathsf{Cycle}_n,Y)$-friendly swap. Observe that $\tau'$ is a linear extension of $\alpha'$. Since $\tau$ and $\tau'$ are adjacent in $\FS(\mathsf{Cycle}_n,Y)$, they are certainly in the same connected component.   
\end{proof}

\begin{corollary}\label{CorApprox}
Let $G$ be a graph on $n$ vertices. Given $\alpha\in\Acyc(G)$, let $\vec{\alpha}$ be an acyclic orientation obtained from $\alpha$ by flipping a source into a sink. The equivalence class $[\vec{\alpha}]_\approx$ depends on only the equivalence class $[\alpha]_\approx$, not the specific representative $\alpha$ or the source in $\alpha$ that is flipped to obtain $\vec{\alpha}$. Therefore, the map $\Phi\colon \Acyc(G)/\!\!\approx\,\to\Acyc(G)/\!\!\approx$ defined by $\Phi([\alpha]_\approx)=[\vec{\alpha}]_\approx$ is well-defined.   
\end{corollary}

\begin{proof}
Let $Y$ be a graph with vertex set $[n]$ such that $\overline Y$ is isomorphic to $G$. 
Choose $\alpha,\alpha'\in\Acyc(\overline Y)$ such that $\alpha\approx\alpha'$. Let $u$ and $u'$ be sources of $\alpha$ and $\alpha'$, respectively. Let $\vec\alpha$ (respectively, $\vec\alpha'$) be the acyclic orientation obtained from $\alpha$ (respectively, $\alpha'$) by flipping $u$ (respectively, $u'$) into a sink.
There exist $\sigma\in\mathcal L(\alpha)$ and $\sigma'\in\mathcal L(\alpha')$ such that $\sigma(1)=u$ and $\sigma'(1)=u'$. Because $\alpha\approx\alpha'$, it follows from Theorem~\ref{Thm2} that $\sigma$ and $\sigma'$ are in the same connected component of $\FS(\mathsf{Cycle}_n,Y)$. Because $\varphi^*$ is an automorphism of $\FS(\mathsf{Cycle}_n,Y)$, the permutations $\varphi^*(\sigma)$ and $\varphi^*(\sigma')$ are in the same connected component of $\FS(\mathsf{Cycle}_n,Y)$. We now check that $\varphi^*(\sigma)\in\mathcal L(\vec{\alpha})$ and $\varphi^*(\sigma')\in\mathcal L(\vec{\alpha}')$. Using Theorem~\ref{Thm2} once again, we find that $\vec{\alpha}\approx\vec{\alpha}'$. This proves that the map $\Phi$ is well-defined. 
\end{proof}

\begin{example}
Figures~\ref{Fig2} and \ref{Fig3} show toric acyclic orientations in red boxes and double-flip equivalence classes in blue boxes. In each case, the double-flip equivalence classes within each toric acyclic orientation are cyclically ordered from left to right. The map $\Phi$ sends each double-flip equivalence class to the next double-flip equivalence class in this cyclic order. 
\end{example}

Theorem~\ref{Thm2} describes the connected components of $\FS(\mathsf{Cycle}_n,Y)$, but it is possible to say even more. Namely, we will obtain a description of these connected components that relies on only the equivalence relation $\sim$, not the double-flip equivalence relation $\approx$. We first handle the case in which $\overline Y$ is connected, where we will see that each toric acyclic orientation of $\overline Y$ corresponds to a union of $n$ pairwise isomorphic connected components of $\FS(\mathsf{Cycle}_n,Y)$. We will then use this result to understand the connected components of $\FS(\mathsf{Cycle}_n,Y)$ when $\overline Y$ is not necessarily connected. 

\begin{proposition}\label{PropConnected}
Let $Y$ be a graph on the vertex set $[n]$ such that $\overline{Y}$ is connected. For each toric acyclic orientation $[\alpha]_{\sim}\in\Acyc(\overline Y)/\!\!\sim$, choose a linear extension $\sigma_{[\alpha]_{\sim}}$ of $[\alpha]_{\sim}$, and let $J_{[\alpha]_{\sim}}$ be the connected component of $\FS(\mathsf{Cycle}_n,Y)$ containing $\sigma_{[\alpha]_{\sim}}$. The graphs \[J_{[\alpha]_{\sim}},\varphi^*(J_{[\alpha]_{\sim}}),\ldots,(\varphi^*)^{n-1}(J_{[\alpha]_{\sim}})\] are distinct, pairwise isomorphic connected components of $\FS(\mathsf{Cycle}_n,Y)$. Moreover, 
\[\FS(\mathsf{Cycle}_n,Y)=\bigoplus_{[\alpha]_{\sim}\in\Acyc(\overline Y)/\sim}\bigoplus_{k=0}^{n-1}(\varphi^*)^k(J_{[\alpha]_{\sim}}).\]
\end{proposition}

\begin{proof}
The claim that the graphs $J_{[\alpha]_{\sim}},\varphi^*(J_{[\alpha]_{\sim}}),\ldots,(\varphi^*)^{n-1}(J_{[\alpha]_{\sim}})$ are pairwise isomorphic connected components of $\FS(\mathsf{Cycle}_n,Y)$ follows immediately from the fact that $\varphi^*$ is an automorphism of $\FS(\mathsf{Cycle}_n,Y)$ (by Proposition~\ref{Prop3}). Each equivalence class in $\Acyc(\overline Y)/\!\!\sim$ is a union of equivalence classes in $\Acyc(\overline Y)/\!\!\approx$, so it follows from Theorem~\ref{Thm2} that \[\FS(\mathsf{Cycle}_n,Y)=\bigoplus_{[\alpha]_{\sim}\in\Acyc(\overline Y)/\sim}\FS(\mathsf{Cycle}_n,Y)\vert_{\mathcal L([\alpha]_\sim)}.\] Let us fix a toric acyclic orientation $[\alpha]_\sim\in\Acyc(\overline Y)/\!\!\sim$ and a linear extension $\sigma_{[\alpha]_\sim}\in\mathcal L([\alpha]_\sim)$. Let $U=\bigcup_{k=0}^{n-1}V((\varphi^*)^k(J_{[\alpha]_\sim}))$. We have seen that applying $\varphi^*$ to a linear extension of an acyclic orientation $\beta\in\Acyc(\overline Y)$ produces a new permutation whose corresponding acyclic orientation is obtained from $\beta$ by applying a flip; it follows that $(\varphi^*)^k(\sigma_{[\alpha]_\sim})$ is a linear extension of $[\alpha]_\sim$ for every $0\leq k\leq n-1$. In particular, $(\varphi^*)^k(\sigma_{[\alpha]_\sim})\in V((\varphi^*)^k(J_{[\alpha]_\sim}))\cap\mathcal L([\alpha]_\sim)$ for every $0\leq k\leq n-1$. Because $\FS(\mathsf{Cycle}_n,Y)\vert_{\mathcal L([\alpha]_\sim)}$ is a union of connected components of $\FS(\mathsf{Cycle}_n,Y)$, we must have $U\subseteq\mathcal L([\alpha]_\sim)$. We need to show that the reverse inclusion $\mathcal L([\alpha]_\sim)\subseteq U$ holds and that the subgraphs $J_{[\alpha]_{\sim}},\varphi^*(J_{[\alpha]_{\sim}}),\ldots,(\varphi^*)^{n-1}(J_{[\alpha]_{\sim}})$ are distinct.

Define a \dfn{move} to be an operation that changes a permutation $\tau$ into a permutation $\tau'$ such that either $\{\tau,\tau'\}\in E(\FS(\mathsf{Cycle}_n,Y))$ or $\tau'=\varphi^*(\tau)$. Say two permutations are \dfn{move-equivalent} if there is a sequence of moves transforming the first permutation into the second. Using the fact that $(\varphi^*)^{-1}=(\varphi^*)^{n-1}$, we see that move-equivalence is a genuine equivalence relation. Applying a move to an element of $U$ produces another element of $U$. Therefore, in order to prove that $\mathcal L([\alpha]_\sim)\subseteq U$, it suffices to show that any two permutations in $\mathcal L([\alpha]_\sim)$ are move-equivalent. If two permutations in $\mathcal L([\alpha]_\sim)$ are linear extensions of the same acyclic orientation of $\overline Y$, then they are certainly move-equivalent because they are in the same connected component of $\FS(\mathsf{Cycle}_n,Y)$ by Theorem~\ref{Thm2}. Therefore, it suffices to show that for any two acyclic orientations $\alpha',\alpha''\in[\alpha]_\sim$, there exist linear extensions $\sigma'\in\mathcal L(\alpha')$ and $\sigma''\in\mathcal L(\alpha'')$ that are move-equivalent. In order to prove this, it suffices to prove it in the case where $\alpha''$ is obtained from $\alpha'$ via a flip. Without loss of generality, we may assume $\alpha''$ is obtained from $\alpha'$ by flipping a source $u$ into a sink (otherwise, switch the roles of $\alpha'$ and $\alpha''$). It is straightforward to see that there exists a linear extension $\sigma'$ of $\alpha'$ such that $\sigma'(1)=u$. The permutation $\sigma''=\varphi^*(\sigma')$ is a linear extension of $\alpha''$. By definition, $\sigma'$ and $\sigma''$ are move-equivalent. 

It remains to prove that $J_{[\alpha]_{\sim}},\varphi^*(J_{[\alpha]_{\sim}}),\ldots,(\varphi^*)^{n-1}(J_{[\alpha]_{\sim}})$ are distinct. This will follow if we can show that for every $\sigma\in\mathfrak S_n$ and every $k\in[n-1]$, the vertices $\sigma$ and $(\varphi^*)^k(\sigma)$ lie in different connected components of $\FS(\mathsf{Cycle}_n,Y)$. To prove this, it is helpful to imagine the cycle graph as lying in the plane with the vertices $1,\ldots,n$ listed clockwise in this order. As in the introduction, we imagine $n$ people labeled with the numbers $1,\ldots,n$ standing on the vertices of thecycle graph, with the person labeled $\sigma(i)$ on the vertex $i$. Two people are friends with each other if and only if their labels are adjacent in $Y$, and two friends can swap places with each other whenever they are at adjacent positions in the cycle. 

Now suppose by way of contradiction that there is a sequence of $(\mathsf{Cycle}_n,Y)$-friendly swaps that changes the configuration of people given by $\sigma$ to the configuration given by $(\varphi^*)^k(\sigma)$. The overall effect of this transformation is that each person moves $k$ spaces counterclockwise. Imagine assigning a weight to each person so that a person's weight increases by $1$ every time they move one space counterclockwise and decreases by $1$ every time they move one space clockwise. Assume each person's weight starts at $0$. After the sequence of $(\mathsf{Cycle}_n,Y)$-friendly swaps that has the effect of moving each person $k$ spaces counterclockwise, each person's weight will be congruent to $k$ modulo $n$. Since each swap increases one person's weight by $1$ and decreases another person's weight by $1$, the total weight of all people is always $0$. Therefore, at least one person must have a positive weight in the end, and at least one person must have a negative weight in the end. Observe that no person has weight $0$ at the end because $k\not\equiv 0\pmod n$. Let $A^+$ be the (nonempty) set of people whose weights are positive in the end, and let $A^-$ be the (nonempty) set of people whose weights are negative in the end. It is clear that if $p^+\in A^+$ and $p^-\in A^-$, then at some point during the sequence of swaps, the person $p^+$ must have swapped places with $p^-$. This means that every person in $A^+$ is friends with every person in $A^-$. Using the labels to identify the people with vertices of $Y$, we see that every person in $A^+$ is adjacent in $Y$ to every person in $A^-$. Since $A^+\cup A^-=V(Y)$, this contradicts the assumption that $\overline Y$ is connected. 
\end{proof} 

\begin{example}
Suppose $Y=\begin{array}{l}\includegraphics[height=.7cm]{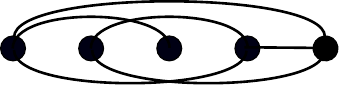}\end{array}$ so that $\overline Y$ is the graph $\begin{array}{l}\includegraphics[height=.4cm]{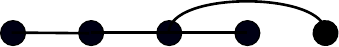},\end{array}$ whose acyclic orientations are depicted in Figure~\ref{Fig3}. Identify the vertices of $Y$, read from left to right, with the numbers $1,2,3,4,5$. Theorem~\ref{Thm2} tells us that the connected components of $\FS(\mathsf{Cycle}_5,Y)$ correspond to the double-flip equivalence classes of $\overline Y$ (the blue boxes in Figure~\ref{Fig3}). There is only $1$ toric acyclic orientation $[\alpha]_\sim$ in $\Acyc(\overline Y)/\!\!\sim$, so we can choose $\sigma_{[\alpha]_\sim}=12345$. In this case, $J_{[\alpha]_\sim}$ is the connected component of $\FS(\mathsf{Cycle}_5,Y)$ corresponding to the leftmost double-flip equivalence class in Figure~\ref{Fig3}. The set of vertices in $J_{[\alpha]_\sim}$ is the set of linear extensions of that double-flip equivalence class, which is 
\begin{equation}\label{Eq4}
\{12354,12345,52341,25341,52314,25314,52134,25134, 21534,54312,45312,41532,
\end{equation}
\[\,\,\,\,\,\,54132,14532, 45132, 51432, 15432, 24351, 42315, 24315, 42135, 24135, 21435, 42351\}.\]
The second double-flip equivalence class from the left, which is obtained by applying the map $\Phi$ from Corollary~\ref{CorApprox} to the leftmost double-flip equivalence class, also has $24$ linear extensions; these are precisely the permutations obtained by applying $\varphi^*$ to the permutations in \eqref{Eq4}. In general, applying $\Phi$ to a double-flip equivalence class $[\beta]_\approx$ yields the next double-flip equivalence class after $[\beta]_\approx$ in the left-to-right cyclic order; the linear extensions of $\Phi([\beta]_{\approx})$ are obtained by applying $\varphi^*$ to the linear extensions of $[\beta]_\approx$.  
\end{example}

The following corollary tells us about an interesting relationship between the sets $\Acyc(G)/\!\!\sim$ and $\Acyc(G)/\!\!\approx$ for an arbitrary graph $G$. It is worth noting that the proof of this fact, which is not obvious \emph{a priori}, passes through the analysis of the graph $\FS(\mathsf{Cycle}_n,Y)$ with $Y=\overline G$. 

\begin{corollary}\label{CorPhiOrbits}
Let $G$ be a connected graph on $n$ vertices, and let $\Phi\colon \Acyc(G)/\!\!\approx\,\to\Acyc(G)/\!\!\approx$ be the map from Corollary~\ref{CorApprox}. For each $\alpha\in\Acyc(G)$, the equivalence class $[\alpha]_\sim$ is the disjoint union \[[\alpha]_\sim=\bigsqcup_{k=0}^{n-1}\Phi^k([\alpha]_\approx).\]  
\end{corollary}

\begin{proof}
Let $Y=\overline G$. We saw in the proof of Corollary~\ref{CorApprox} that applying $\Phi$ to an equivalence class $[\beta]_\approx$ corresponds to applying $\varphi^*$ to the linear extensions of $\beta$. Preserving the notation from Proposition~\ref{PropConnected}, we can choose $\sigma_{[\alpha]_\sim}$ to be a linear extension of $\alpha$ so that $[\alpha]_\approx$ is the set of vertices of $J_{[\alpha]_\sim}$. Therefore, the desired result is equivalent to the part of Proposition~\ref{PropConnected} stating that \[\FS(\mathsf{Cycle}_n,Y)\vert_{\mathcal L([\alpha]_\sim)}=\bigoplus_{k=0}^{n-1}(\varphi^*)^k(J_{[\alpha]_\sim}). \qedhere\] 
\end{proof}

We can now proceed to describe the connected components of $\FS(\mathsf{Cycle}_n,Y)$ without assuming that $\overline Y$ is connected. The reader may find it helpful to refer to Example~\ref{Exam1} and Figure~\ref{Fig4} while reading the following proof.   

\begin{theorem}\label{Thm5}
Let $Y$ be a graph on the vertex set $[n]$.  Let $n_1, \ldots, n_r$ denote the sizes of the connected components of $\overline{Y}$, and let $\nu=\gcd(n_1, \ldots, n_r)$.   For each toric acyclic orientation $[\alpha]_{\sim}\in\Acyc(\overline Y)/\!\!\sim$, choose a linear extension $\sigma_{[\alpha]_{\sim}}$ of $[\alpha]_{\sim}$, and let $J_{[\alpha]_{\sim}}$ be the connected component of $\FS(\mathsf{Cycle}_n,Y)$ containing $\sigma_{[\alpha]_{\sim}}$. The graphs $J_{[\alpha]_{\sim}},\varphi^*(J_{[\alpha]_{\sim}}),\ldots,(\varphi^*)^{\nu-1}(J_{[\alpha]_{\sim}})$ are distinct, pairwise isomorphic connected components of $\FS(\mathsf{Cycle}_n,Y)$. Moreover, 
\[\FS(\mathsf{Cycle}_n,Y)=\bigoplus_{[\alpha]_{\sim}\in\Acyc(\overline Y)/\sim}\bigoplus_{k=0}^{\nu-1}(\varphi^*)^k(J_{[\alpha]_{\sim}}).\]
\end{theorem}

\begin{proof}
As in the proof of Proposition~\ref{PropConnected}, the assertions that $J_{[\alpha]_{\sim}},\varphi^*(J_{[\alpha]_{\sim}}),\ldots,(\varphi^*)^{\nu-1}(J_{[\alpha]_{\sim}})$ are pairwise isomorphic connected components of $\FS(\mathsf{Cycle}_n,Y)$ and that
\[\FS(\mathsf{Cycle}_n,Y)=\bigoplus_{[\alpha]_{\sim}\in\Acyc(\overline Y)/\sim}\FS(\mathsf{Cycle}_n,Y)\vert_{\mathcal L([\alpha]_\sim)}\] are immediate. Let us fix a toric acyclic orientation $[\alpha]_\sim\in\Acyc(\overline Y)/\!\!\sim$ and a linear extension $\sigma_{[\alpha]_\sim}\in\mathcal L([\alpha]_\sim)$. We may assume that $\sigma_{[\alpha]_\sim}$ is actually a linear extension of the specific acyclic orientation $\alpha$. The same argument used in the proof of Proposition~\ref{PropConnected} shows that the vertex sets of $J_{[\alpha]_{\sim}},\varphi^*(J_{[\alpha]_{\sim}}),\ldots,(\varphi^*)^{\nu-1}(J_{[\alpha]_{\sim}})$ are all contained in $\mathcal L([\alpha]_\sim)$. We need to show that the subgraphs $J_{[\alpha]_{\sim}},\varphi^*(J_{[\alpha]_{\sim}}),\ldots,(\varphi^*)^{\nu-1}(J_{[\alpha]_{\sim}})$ are distinct and that every element of $\mathcal L([\alpha]_\sim)$ is a vertex of one of these $\nu$ subgraphs. Since $J_{[\alpha]_\sim}$ is the connected component of $\FS(\mathsf{Cycle}_n,Y)$ containing $\sigma_{[\alpha]_\sim}$ (which is a linear extension of $\alpha$), it follows from Theorem~\ref{Thm2} that $V(J_{[\alpha]_\sim})=\mathcal L([\alpha]_{\approx})$. The proof of Corollary~\ref{CorApprox} tells us that the vertex set of $(\varphi^*)^k(J_{[\alpha]_\sim})$ is the set of linear extensions of $\Phi^k([\alpha]_\approx)$. Therefore, we will be done if we can prove that $[\alpha]_\sim$ decomposes into the disjoint union 
\begin{equation}\label{Eq3}
[\alpha]_\sim=\bigsqcup_{k=0}^{\nu-1}\Phi^k([\alpha]_{\approx}).
\end{equation}

Let $Z_1, \ldots, Z_r$ denote the connected components of $\overline{Y}$, where each $Z_i$ has size $n_i$.  Let \[\Phi_i\colon \Acyc(Z_i)/\!\!\approx\,\to\Acyc(Z_i)/\!\!\approx\] be the map that flips a source of an acyclic orientation of $Z_i$ into a sink, which is well-defined by Corollary~\ref{CorApprox}. We now define yet another equivalence relation on $\Acyc(\overline{Y})$.  Say that a double flip is a \dfn{local double flip} if the flipped source and sink are in the same connected component of $\overline Y$.  We say that two acyclic orientations $\beta, \beta' \in \Acyc(\overline{Y})$ are \dfn{local-double-flip equivalent}, denoted $\beta \tripprox \beta'$, if $\beta'$ can be obtained from $\beta$ by a sequence of local double flips. Note that every equivalence class in $\Acyc(\overline{Y})/\!\!\approx$ is a union of equivalence classes in $\Acyc(\overline{Y})/\!\!\tripprox$. 

Every acyclic orientation $\beta$ of $\overline Y$ restricts to an acyclic orientaion $\beta^{(i)}$ of $Z_i$. This yields a map $f\colon \Acyc(\overline Y)\to\prod_{i=1}^r\Acyc(Z_i)$ given by $f(\beta)=(\beta^{(1)},\ldots,\beta^{(r)})$, which induces the bijections 
\begin{equation}\label{Eq1}f_1:\Acyc(\overline Y)/\!\!\sim\,\longrightarrow\,\prod_{i=1}^r(\Acyc(Z_i)/\!\!\sim)
\end{equation}
and
\begin{equation}\label{Eq2}
f_2:\Acyc(\overline Y)/\!\!\tripprox\,\longrightarrow\prod_{i=1}^r(\Acyc(Z_i)/\!\!\approx).
\end{equation}
In particular, the bijection $f_1$ maps our fixed toric acyclic orientation $[\alpha]_{\sim}$ to the tuple of toric acyclic orientations $([\alpha^{(1)}]_{\sim},\ldots,[\alpha^{(r)}]_{\sim})$.  

It follows from Corollary~\ref{CorPhiOrbits} that each toric acyclic orientation $[\alpha^{(i)}]_{\sim}$ on $Z_i$ consists of the disjoint union of $n_i$ equivalence classes in $\Acyc(Z_i)/\!\!\approx$ and that there is a faithful transitive action of  $\mathbb{Z}/n_i\mathbb{Z}$ on these $n_i$ equivalence classes given by $\Phi_i$.  By combining these individual cyclic actions and using the bijection $f_2$, we obtain a faithful transitive action of the (additive) group $\Gamma=(\mathbb{Z}/n_1\mathbb{Z}) \times \cdots \times (\mathbb{Z}/n_r\mathbb{Z})$ on the set of $\tripprox$-equivalence classes contained in $[\alpha]_\sim$.
In particular, $[\alpha]_{\sim}$ is the disjoint union of $n_1\cdots n_r$ equivalence classes in $\Acyc(\overline{Y})/\!\!\tripprox$. 

Let $\gamma_i$ denote the element of $\Gamma$ that acts by applying $\Phi_i$ to $\Acyc(Z_i)/\!\!\approx$. Every time we apply a local double flip to an acyclic orientation of $\overline Y$, we do not change the $\tripprox$-equivalence class. Every time we apply a double flip that is not a local double flip, we must flip a source in some connected component $Z_i$ into a sink and simultaneously flip a sink in some different connected component $Z_j$ into a source. This non-local double flip has the effect of applying the action of $\gamma_i-\gamma_j$ to the $\tripprox$-equivalence class. Let $\Delta$ be the subgroup of $\Gamma$ generated by all elements of the form $\gamma_i-\gamma_j$; we see that each element of $\Acyc(\overline{Y})/\!\!\approx$ contained in $[\alpha]_\sim$ is the union of a $\Delta$-orbit of $\Acyc(\overline{Y})/\!\!\tripprox$. Furthermore, for every element of $\Acyc(\overline{Y})/\!\!\tripprox$ contained in $[\alpha]_\sim$, the union of the sets in its $\Delta$-orbit is an element of $\Acyc(\overline{Y})/\!\!\approx$ contained in $[\alpha]_\sim$. It is now straightforward to compute (using B\'ezout's Lemma, say) that $\Delta$ has index $\nu$ in $\Gamma$ and that the cosets in $\Gamma/\Delta$ are $k\gamma_1+\Delta$ for $0 \leq k \leq \nu-1$. The element $\gamma_1$ acts by applying $\Phi_1$, which changes a source in $Z_1$ into a sink.
If we have an equivalence class $[\beta]_\approx$ contained in $[\alpha]_\sim$, then changing a source of $
\beta$ in $Z_1$ into a sink (i.e., applying $\Phi_1$ to $[\beta^{(1)}]_\approx$) corresponds to applying $\Phi$ to $[\beta]_\approx$. Therefore, \eqref{Eq3} follows from the Orbit-Stabilizer Theorem. 
\end{proof}

\begin{example}\label{Exam1}
Suppose $Y=\begin{array}{l}\includegraphics[height=.7cm]{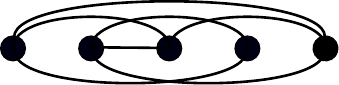}\end{array}$ so that $\overline Y$ is the graph $\begin{array}{l}\includegraphics[height=.22cm]{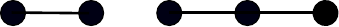}.\end{array}$ Let $Z_1$ be the connected component of $\overline Y$ with $2$ vertices, and let $Z_2$ be the connected component with $3$ vertices. Each green box in Figure~\ref{Fig4} encompasses a local-double-flip equivalence class of acyclic orientations of $\overline Y$; this figure shows how $\Phi_1$ and $\Phi_2$ act on these equivalence classes. In this case, the group $\Gamma$ is $(\mathbb Z/2\mathbb Z)\times(\mathbb Z/3\mathbb Z)$. The element $\gamma_1=(1,0)$ acts on the double-flip equivalence classes via $\Phi_1$, and $\gamma_2=(0,1)$ acts via $\Phi_2$. The graph $\overline Y$ has only $1$ toric acyclic orientation. Because $\gcd(2,3)=1$, it follows from Theorem~\ref{Thm5} that $\FS(\mathsf{Cycle}_5,Y)$ has only $1$ connected component; according to Theorem~\ref{Thm2}, this implies that there is only $1$ double-flip equivalence class of acyclic orientations of $\overline Y$. Indeed, this is because $\Gamma$ is generated by the element $\gamma_1-\gamma_2=(1,-1)$. We can reach any local-double-flip equivalence class from any other by repeatedly applying $\gamma_1-\gamma_2$, and each application of $\gamma_1-\gamma_2$ corresponds to applying a non-local double flip.   

\begin{figure}[ht]
\begin{center}
\includegraphics[width=.8\linewidth]{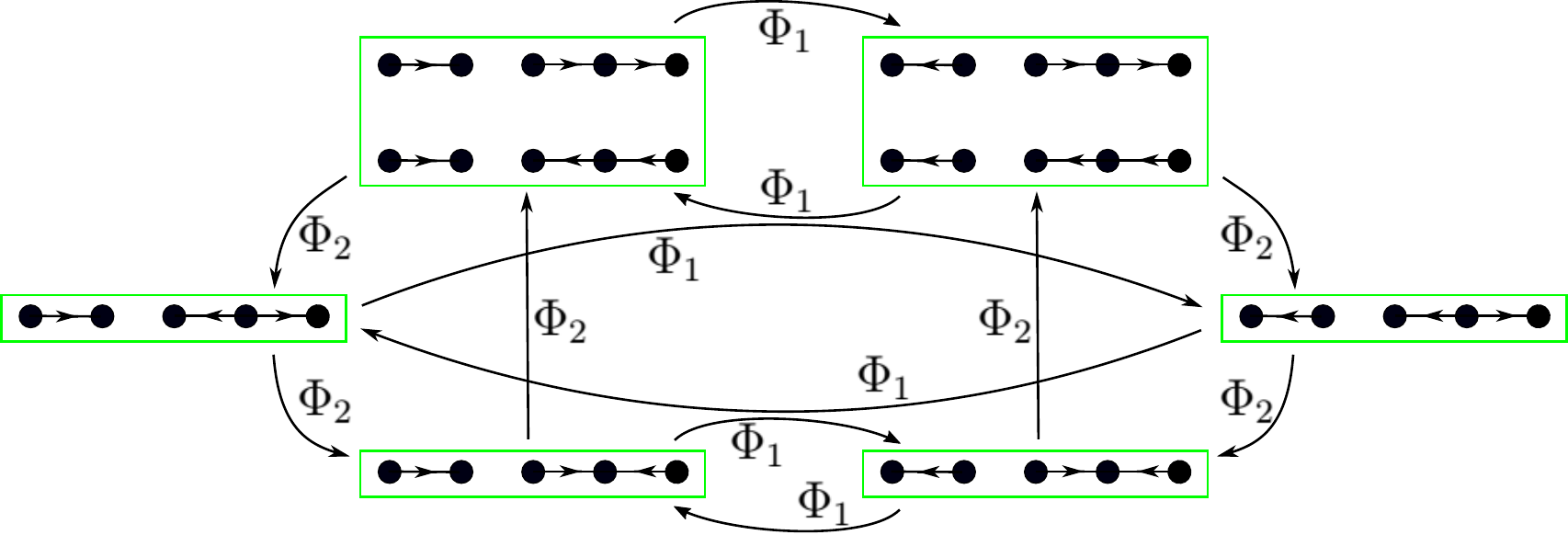}
\caption{The action of $\Phi_1$ and $\Phi_2$ on local-double-flip equivalence classes.}
\label{Fig4}
\end{center}  
\end{figure}
\end{example}
 
\begin{remark}
Suppose $G$ is a graph with connected components of sizes $n_1,\ldots,n_r$, and let $\nu=\gcd(n_1,\ldots,n_r)$. Letting $Y=\overline G$, we see immediately from Theorem~\ref{Thm2} and Theorem~\ref{Thm5} that each toric acyclic orientation of $G$ is a disjoint union of exactly $\nu$ double-flip equivalence classes. \emph{A priori}, it is not obvious that this should be the case. 
\end{remark}

Theorem~\ref{Thm5} allows us to enumerate the connected components of $\FS(\mathsf{Cycle}_n,Y)$ in terms of the Tutte polynomial $T_{\overline Y}(x,y)$ of $\overline Y$; this is analogous to Corollary~\ref{Cor1}. 

\begin{corollary}
Let $Y$ be a graph with $n\geq 3$ vertices. Let $Z_1,\ldots,Z_r$ be the connected components of $\overline Y$, and let $\nu=\gcd(|V(Z_1)|,\ldots,|V(Z_r)|)$. Then the number of connected components of $\FS(\mathsf{Cycle}_n,Y)$ is $T_{\overline Y}(1,0)\nu$. 
\end{corollary}

\begin{proof}
It is known \cite{Develin} that $T_{\overline Y}(1,0)$ is the number of toric acyclic orientations of $\overline Y$, so the enumeration of the connected components follows from Theorem~\ref{Thm5}. 
\end{proof}

\begin{remark}
If $\overline Y$ is connected, then the number of connected components of $\FS(\mathsf{Cycle}_n,Y)$ is $T_{\overline Y}(1,0)n$. It follows from a well-known theorem of Greene and Zaslavsky \cite{Greene} that this is also the number of acyclic orientations of $\overline Y$ with exactly $1$ source. 
\end{remark}

For an even more concrete application of Theorem~\ref{Thm5}, we consider the special case in which $\overline Y$ is a forest. 
\begin{corollary}\label{Cor3}
Let $Y$ be a graph with $n\geq 3$ vertices such that $\overline Y$ is a forest consisting of trees $\mathscr T_1,\ldots,\mathscr T_r$, and let $\nu=\gcd(|V(\mathscr T_1)|,\ldots,|V(\mathscr T_r)|)$. Then $\FS(\mathsf{Cycle}_n,Y)$ has $\nu$ connected components. These connected components are pairwise isomorphic, and they each contain $n!/\nu$ vertices. Moreover, each connected component of $\FS(\mathsf{Cycle}_n,Y)$ has an automorphism of order $2$. 
\end{corollary}

\begin{proof}
We may assume $V(Y)=[n]$. It is known that $T_{\overline Y}(1,0)=\lvert\Acyc(\overline Y)/\!\!\sim\rvert=1$ because $\overline Y$ is a forest. Consequently, the statements about the enumeration and sizes of the connected components of $\FS(\mathsf{Cycle}_n,Y)$ follow directly from Theorem~\ref{Thm5}. 

Let $H$ be one of the connected components of $\FS(\mathsf{Cycle}_n,Y)$. By Theorem~\ref{Thm5}, the connected components of $\FS(\mathsf{Cycle}_n,Y)$ are $H,\varphi^*(H),\ldots,(\varphi^*)^{\nu-1}(H)$. We want to show that each connected component $(\varphi^*)^i(H)$ has an automorphism of order $2$; since these connected components are pairwise isomorphic, it suffices to prove that $H$ has an automorphism of order $2$. The automorphism group of $\mathsf{Cycle}_n$ is the dihedral group of order $2n$; let $\psi$ be one of the reflections in this group. Let $m\in\{0,\ldots,\nu-1\}$ be such that $\psi^*(H)=(\varphi^*)^m(H)$ (where, by abuse of notation, $\psi^*(H)$ denotes the connected component whose vertices form the set $\psi^*(V(H))$). Then $(\psi\circ\varphi^{-m})^*(H)=(\varphi^*)^{-m}(\psi^*(H))=H$. Since $\psi\circ\varphi^{-m}$ is a reflection in the dihedral group of order $2n$, it has order $2$. 
\end{proof}

\begin{example}\label{Exam3}
Let us revisit Example~\ref{Exam2}, which concerns the graph $\FS(\mathsf{Cycle}_5,Y)$ where \[Y=\begin{array}{l}\includegraphics[height=1.6cm]{FriendsPIC17}\end{array}.\]  Notice that the complement of $Y$ is the tree \[\overline Y=\begin{array}{l}\includegraphics[height=1.6cm]{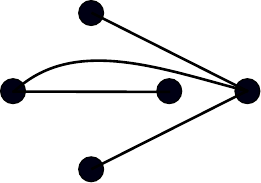}\end{array};\] Corollary~\ref{Cor3} now explains why $\FS(\mathsf{Cycle}_5,Y)$ has $5$ pairwise isomorphic connected components, each with an automorphism of order $2$.
\end{example}

\begin{corollary}\label{Cor2}
Let $Y$ be a graph with $n\geq 3$ vertices. The graph $\FS(\mathsf{Cycle}_n,Y)$ is connected if and only if $\overline Y$ is a forest consisting of trees $\mathscr T_1,\ldots,\mathscr T_r$ such that $\gcd(|V(\mathscr T_1)|,\ldots,|V(\mathscr T_r)|)=1$. 
\end{corollary}

\begin{proof}
If $\overline Y$ is a forest consisting of trees $\mathscr T_1,\ldots,\mathscr T_r$, then Corollary~\ref{Cor3} tells us that $\FS(\mathsf{Cycle}_n,Y)$ is connected if and only if $\gcd(|V(\mathscr T_1)|,\ldots,|V(\mathscr T_r)|)=1$. Now suppose $\overline Y$ is not a forest, meaning it contains a cycle of some length $m$ with $3\leq m\leq n$. Let $Q$ be the graph consisting of $n-m$ isolated vertices and one cycle of length $m$. Since $Q$ is a subgraph of $\overline Y$, the graph $Y$ must be a subgraph of $\overline Q$. Our goal is to show that $\FS(\mathsf{Cycle}_n,Y)$ is disconnected; by Proposition~\ref{Prop1}, it suffices to prove that $\FS(\mathsf{Cycle}_n,\overline Q)$ is disconnected. By Theorem~\ref{Thm5}, it suffices to show that $Q$ has at least $2$ toric acyclic orientations. As mentioned above, the number of toric acyclic orientations of $Q$ is $T_{Q}(1,0)$, where $T_Q(x,y)$ is the Tutte polynomial of $Q$. It is well known (and easy to prove from definitions) that $T_Q(x,y)=(x+x^2+\cdots+x^{m-1})+y$, so $T_Q(1,0)=m-1\geq 2$, as desired.  
\end{proof}

Recall that a graph on $n$ vertices is called \dfn{Hamiltonian} if it has a subgraph isomorphic to $\mathsf{Cycle}_n$. By combining Corollary~\ref{Cor2} with Proposition~\ref{Prop1}, we obtain a sufficient condition to guarantee that $\FS(X,Y)$ is connected whenever $X$ is Hamiltonian: if $X$ and $Y$ are graphs on $n$ vertices such that $X$ is Hamiltonian and $\overline Y$ is a forest consisting of trees $\mathscr T_1,\ldots,\mathscr T_r$ with $\gcd(|V(\mathscr T_1)|,\ldots,|V(\mathscr T_r)|)=1$, then $\FS(X,Y)$ is connected.

\section{Sufficient Conditions for Connectivity}\label{Sec:Hereditary}

A \dfn{hereditary class} is a collection of (isomorphism types of) graphs that is closed under taking induced subgraphs. For example, for each fixed $d\geq 1$, the set of all graphs with maximum degree at most $d$ is hereditary. Another example of a hereditary class is the set of all bipartite graphs. Recall that if $G$ is a graph on $n$ vertices, then a \emph{Hamiltonian path} in $G$ is a subgraph of $G$ isomorphic to $\mathsf{Path}_n$.

Our goal in this section is to find sufficient conditions on the graphs $X$ and $Y$ that guarantee $\FS(X,Y)$ being connected. Our general setup will involve some fixed hereditary class $\mathcal H$. We will consider the case in which $X$ contains a Hamiltonian path and $\overline Y$ belongs to $\mathcal H$. We will state our main results in fairly general terms and then exhibit several specific applications. Our results rely heavily on Theorem~\ref{Thm1}, which tells us that the connected components of $\FS(\mathsf{Path}_n,Y)$ correspond to acyclic orientations of $\overline Y$. It will be convenient to assume that $X$ has vertex set $[n]$ and that $\mathsf{Path}_n$ is a genuine subgraph of $X$ (meaning $\{i,i+1\}\in E(X)$ for all $i\in[n-1]$). This does not sacrifice any generality because the isomorphism type of $\FS(X,Y)$ depends only on the isomorphism types of $X$ and $Y$. Proposition~\ref{Prop1} tells us that the vertex set of each connected component of $\FS(X,Y)$ is a union of vertex sets of connected components of $\FS(\mathsf{Path}_n,Y)$; Theorem~\ref{Thm1} tells us that the vertex set of a connected component of $\FS(\mathsf{Path}_n,Y)$ is the set of linear extensions of the corresponding acyclic orientation of $\overline Y$.  

In what follows, recall that if $G$ is a graph with vertex set $[n]$ and $\sigma\in\mathfrak S_n$, then $\alpha_G(\sigma)$ denotes the unique acyclic orientation of $G$ that has $\sigma$ as a linear extension. 

\begin{theorem}\label{thm:hereditary}
Let $X$ and $Y$ be graphs with $V(X)=V(Y)=[n]$, and suppose that $\mathsf{Path}_n$ is a subgraph of $X$. Suppose also that each connected component $B$ of $\FS(X,Y)$ contains some permutation $\sigma_B$ such that $n$ is a sink of $\alpha_{\overline Y}(\sigma_B)$.  Then the number of connected components of $\FS(X,Y)$ is at most the number of connected components of $\FS(X\vert_{[n-1]}, Y\vert_{[n-1]})$.
\end{theorem}

\begin{proof}
Let $B_1,\ldots,B_r$ be the connected components of $\FS(X,Y)$, and let $C_1, \ldots, C_s$ be the connected components of $\FS(X\vert_{[n-1]}, Y\vert_{[n-1]})$. Given $i\in[r]$, we know by hypothesis that there is a permutation $\sigma_{B_i}$ in the connected component $B_i$ such that $n$ is a sink of $\alpha_{\overline Y}(\sigma_{B_i})$. Theorem~\ref{Thm1} tells us that every linear extension of $\alpha_{\overline Y}(\sigma_{B_i})$ is in $B_i$. By replacing $\sigma_{B_i}$ with a different linear extension of $\alpha_{\overline Y}(\sigma_{B_i})$ that sends $n$ to $n$ (such a linear extension certainly exists), we may assume without loss of generality that $\sigma_{B_i}(n)=n$. The acyclic orientation $\alpha_{\overline Y}(\sigma_{B_i})$ restricts to an acyclic orientation $\beta_{B_i}$ of $\overline Y\vert_{[n-1]}$. Because $\mathsf{Path}_{n-1}$ is a subgraph of $X\vert_{[n-1]}$, Theorem~\ref{Thm1} implies that all linear extensions of $\beta_{B_i}$ belong to a single connected component $f(B_i)$ of $\FS(X\vert_{[n-1]},Y\vert_{[n-1]})$. We will show that the resulting map $f\colon  \{B_1, \ldots, B_r\} \to \{C_1, \ldots, C_s\}$ is injective.

Suppose $f(B_i)=f(B_j)=C_k$.  Then $\alpha_{\overline{Y}}(\sigma_{B_i})$ and $\alpha_{\overline{Y}}(\sigma_{B_j})$ both have $n$ as a sink, and $\mathcal L(\beta_{B_i})$ and $\mathcal L(\beta_{B_j})$ are both subsets of $V(C_k)$.  Because $\sigma_{B_i}$ and $\sigma_{B_j}$ send $n$ to $n$, their restrictions $\sigma'_{B_i}=\sigma_{B_i}\vert_{[n-1]}$ and $\sigma'_{B_j}=\sigma_{B_j}\vert_{[n-1]}$ are linear extensions of $\beta_{B_i}$ and $\beta_{B_j}$ (respectively) and hence are both vertices in $C_k$.  Thus, there is a sequence of $(X\vert_{[n-1]},Y\vert_{[n-1]})$-friendly swaps that transforms $\sigma'_{B_i}$ into $\sigma'_{B_j}$.  This sequence of $(X\vert_{[n-1]},Y\vert_{[n-1]})$-friendly swaps can be viewed as a sequence of $(X,Y)$-friendly swaps that transforms $\sigma_{B_i}$ into $\sigma_{B_j}$, which implies that $B_i=B_j$, as desired.
\end{proof}

We now investigate more carefully the conditions under which $\FS(X,Y)$ is connected.

\begin{theorem}\label{thm:hereditary-1}
Let $X$ and $Y$ be graphs with $V(X)=V(Y)=[n]$, and suppose that $\mathsf{Path}_n$ is a subgraph of $X$. Suppose that $Y$ is connected and that for every $(n-1)$-vertex induced subgraph $Y'$ of $Y$, the graph $\FS(X\vert_{[n-1]},Y')$ is connected.  Then $\FS(X,Y)$ is connected.
\end{theorem}

\begin{proof}
The theorem will follow from Theorem~\ref{thm:hereditary} if we can show that each connected component of $\FS(X,Y)$ contains a permutation such that the vertex $n$ is a sink in the associated acyclic orientation of $\overline Y$.  Fix a connected component $B_i$ of $\FS(X,Y)$; we say a vertex $v$ of $\overline Y$ is a \dfn{$\overline Y$-sink relative to $B_i$} if there exists a permutation in $V(B_i)$ whose associated acyclic orientation of $\overline Y$ has $v$ as a sink. Our goal is to prove that $n$ is a $\overline Y$-sink relative to $B_i$; we will actually prove the stronger fact that every vertex in $\overline Y$ is a $\overline Y$-sink relative to $B_i$. First, note that at least one vertex of $\overline Y$ must be a $\overline Y$-sink relative to $B_i$. Indeed, if $\sigma\in V(B_i)$, then $\sigma(n)$ is a $\overline Y$-sink relative to $B_i$. Because $Y$ is connected, it now suffices to prove that if $\{y,z\}\in E(Y)$ and $y$ is a $\overline Y$-sink relative to $B_i$, then $z$ is also a $\overline Y$-sink relative to $B_i$. 

Assume $\{y,z\}\in E(Y)$ and $y$ is a $\overline Y$-sink relative to $B_i$. There is a permutation $\sigma\in V(B_i)$ such that $y$ is a sink of $\alpha_{\overline Y}(\sigma)$. Theorem~\ref{Thm1} tells us that every linear extension of $\alpha_{\overline Y}(\sigma)$ is a vertex in $B_i$, and one such linear extension must send $n$ to $y$. Therefore, we may assume without loss of generality that $\sigma(n)=y$. Let $Y'$ be the induced subgraph of $Y$ on $[n] \setminus\{y\}$. The vertices in $\FS(X\vert_{[n-1]},Y')$ are the bijections from $[n-1]$ to $[n]\setminus \{y\}$, which are just the restrictions to $[n-1]$ of the bijections in $\mathfrak S_n$ that send $n$ to $y$. The hypothesis that $\FS(X \vert_{[n-1]},Y')$ is connected tells us that there is a sequence of $(X \vert_{[n-1]},Y')$-friendly swaps that transforms the vertex $\sigma'=\sigma\vert_{[n-1]}$ into a permutation $\tau'$ satisfying $\tau'(n-1)=z$. Let $\tau$ be the vertex in $\FS(X,Y)$ that agrees with $\tau'$ on $[n-1]$ and satisfies $\tau(n)=y$. In particular, $\tau(n-1)=z$. Since $\{y,z\}$ is not an edge in $\overline Y$, it follows that $z$ must be a sink in $\alpha_{\overline Y}(\tau)$. Finally, the same sequence of $(X \vert_{[n-1]},Y')$-friendly swaps that transforms $\sigma'$ into $\tau'$ can be interpreted as a sequence of $(X,Y)$-friendly swaps that transforms $\sigma$ into $\tau$. This proves that $\tau\in V(B_i)$, so $z$ is a $\overline Y$-sink relative to $B_i$, as desired.
\end{proof}

It will be helpful to have a notion that captures the idea of extending a Hamiltonian path of a graph $X$ and then adding additional edges. Thus, if $X$ is a graph with a Hamiltonian path, then we define a \dfn{prolongation} of $X$ to be a graph $\widetilde{X}$ such that:
\begin{itemize}
\item $\widetilde{X}$ contains a (not necessarily induced) subgraph $X^\#$ that is isomorphic to $X$;
\item $\widetilde{X}$ contains a Hamiltonian path that itself contains a Hamiltonian path of $X^\#$.
\end{itemize}
Theorem~\ref{thm:hereditary-1} has an immediate corollary in the language of hereditary classes and prolongations. 

\begin{corollary}\label{cor:hereditary}
Let $\mathcal{H}$ be a hereditary class.  Let $X$ be a graph on $n_0$ vertices with a Hamiltonian path, and suppose that $\FS(X,Y)$ is connected for every $Y \in \mathcal{H}$ on $n_0$ vertices.  If $\widetilde{X}$ is a prolongation of $X$ with $n$ vertices, then $\FS(\widetilde{X},\widetilde{Y})$ is connected for every $\widetilde{Y} \in \mathcal{H}$ on $n$ vertices.
\end{corollary}

\begin{proof}
The proof is by induction on $n$. The case $n=n_0$ is obvious, so suppose $n>n_0$. By assumption, $\widetilde{X}$ contains a subgraph $X^\#$ that is isomorphic to $X$ and a Hamiltonian path $P$ that contains a Hamiltonian path of $X^\#$. Since $n>n_0$, we can find an endpoint $p$ of the path $P$ such that $\widetilde{X}\vert_{V(\widetilde{X})\setminus\{p\}}$ is a prolongation of $X$ on $n-1$ vertices. Let us identify $V(\widetilde{X})$ with $[n]$ in such a way that $p$ is identified with $n$ and $P$ is identified with $\mathsf{Path}_n$. Then $\widetilde{X}\vert_{[n-1]}$ is a prolongation of $X$. Since $\widetilde{Y}$ is in the hereditary class $\mathcal H$, every $(n-1)$-vertex induced subgraph of $\widetilde{Y}\vert_{V(\widetilde{Y})\setminus\{y\}}$ is in $\mathcal H$. By our induction hypothesis, the graph $\FS(X\vert_{[n-1]},\widetilde{Y}\vert_{V(\widetilde{Y})\setminus\{y\}})$ is connected for every $y\in V(\widetilde{Y})$. If we identify $V(\widetilde{Y})$ with $[n]$ in an arbitrary way, then it follows from Theorem~\ref{thm:hereditary-1} that $\FS(\widetilde{X},\widetilde{Y})$ is connected. 
\end{proof}

The set of all graphs $Y$ such that $\overline{Y}$ has maximum degree at most $d$ forms a natural hereditary class.  Equivalently, this hereditary class consists of all graphs $Y$ with minimum degree at least $|V(Y)|-d-1$.  Given some $d\geq 1$ and a graph $X$ on $n_0$ vertices, it requires only a finite search to determine whether or not $\FS(X,Y)$ is connected for every graph $Y$ on $n_0$ vertices with minimum degree at least $n_0-d-1$.  Carrying out this finite computation for various $d$'s and $X$'s (with computer assistance) gives corollaries for infinite families of graphs.  Each of these results would presumably require significant work to prove independently in an ad hoc way; Corollary~\ref{cor:hereditary} provides a unified framework for understanding them.

\begin{corollary}\label{cor:triangle-orb}
Let $X$ be the triangle $K_3$.  For every prolongation $\widetilde{X}$ of $X$ and for every graph $\widetilde{Y}$ on $|V(\widetilde{X})|$ vertices with minimum degree at least $|V(\widetilde{X})|-2$, the graph $\FS(\widetilde{X},\widetilde{Y})$ is connected.
\end{corollary}

\begin{corollary}\label{cor:weird-1}
Let $X$ be the graph $\begin{array}{l}\includegraphics[height=0.8cm]{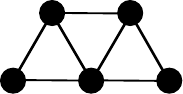}\end{array}$.  For every prolongation $\widetilde{X}$ of $X$ and for every graph $\widetilde{Y}$ on $|V(\widetilde{X})|$ vertices with minimum degree at least $|V(\widetilde{X})|-3$, the graph $\FS(\widetilde{X},\widetilde{Y})$ is connected.
\end{corollary}

\begin{corollary}\label{cor:weird-2}
Let $X$ be the graph $\begin{array}{l}\includegraphics[height=1.33cm]{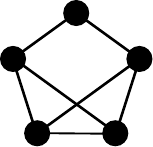}\end{array}$.  For every prolongation $\widetilde{X}$ of $X$ and for every graph $\widetilde{Y}$ on $|V(\widetilde{X})|$ vertices with minimum degree at least $|V(\widetilde{X})|-3$, the graph $\FS(\widetilde{X},\widetilde{Y})$ is connected.
\end{corollary}

Our motivation for investigating the phenomena of this section began with an attempt to understand $\FS(X,Y)$ when $X$ is the graph
\[\begin{array}{l}\includegraphics[height=0.95cm]{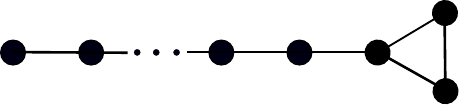}\end{array}.\] More specifically, we are interested in characterizing the graphs $Y$ such that $\FS(X,Y)$ is connected. It follows from Corollary~\ref{cor:triangle-orb} that requiring $Y$ to have minimum degree at least $n-2$ is a sufficient condition to guarantee that $\FS(X,Y)$ is connected. We will see in the next section that this condition is also necessary.

\section{Necessary Conditions for Connectivity}\label{Sec:Lollipops}

We begin by describing a fairly general necessary condition for $\FS(X,Y)$ to be connected in the case where $X$ has a long path consisting of cut vertices and $Y$ has a vertex of low degree.  This result can be understood as an extension of Proposition~\ref{prop5}.

\begin{theorem}\label{thm:cut-path}
Let $X$ and $Y$ be graphs on $n$ vertices.  Suppose $x_1\cdots x_d$ ($d \geq 1$) is a path in $X$, where $x_1$ and $x_d$ are cut vertices and each of $x_2, \ldots, x_{d-1}$ has degree exactly $2$.  If the minimum degree of $Y$ is at most $d$, then $\FS(X,Y)$ is disconnected.
\end{theorem}

\begin{proof}
Proposition~\ref{Prop1} tells us that adding edges to $Y$ cannot make $\FS(X,Y)$ become disconnected, so we may assume that the minimum degree of $Y$ is exactly $d$; let $y_0 \in V(Y)$ be a vertex of degree $d$, and denote its neighbors by $y_1, \ldots, y_d$.  We now identify special vertices $x_0, x_{d+1} \in V(X)$ as follows.  First, suppose $d=1$.  Since $x_1$ is a cut vertex of $X$, it has neighbors in multiple connected components of $X\vert_{V(X)\setminus\{x_1\}}$; let $x_0$ and $x_2$ be neighbors of $x_1$ in different connected components of $X\vert_{V(X)\setminus\{x_1\}}$.  Second, suppose $d>1$.  Since $x_1$ is a cut vertex of $X$, it has a neighbor in a connected component of $X\vert_{V(X)\setminus\{x_1\}}$ that does not contain $x_2$; fix $x_0$ to be any such neighbor of $x_1$.  Similarly, since $x_d$ is a cut vertex of $X$, it has a neighbor in a connected component of $X\vert_{V(X)\setminus\{x_d\}}$ that does not contain $x_{d-1}$; fix $x_{d+1}$ to be any such neighbor of $x_d$.  Note that any path from $x_0$ to $x_{d+1}$ must traverse $x_1\cdots x_d$.  Note also that each of the vertices $x_1,\ldots,x_d$ is a cut vertex of $X$.  Finally, let $R$ denote the vertex set of the connected component of $X\vert_{V(X)\setminus\{x_1\}}$ that contains $x_0$. See Figure~\ref{Fig6}. 

Fix some bijection $\sigma \in \FS(X,Y)$ such that $\sigma(x_0)=y_0$.  Let $B$ be the connected component of $\FS(X,Y)$ that contains $\sigma$.  We claim that every bijection $\tau$ in $B$ satisfies $\tau(x_{d+1}) \neq y_0$; this will imply that $\FS(X,Y)$ has multiple connected components.  We will in fact establish the following stronger statement: If $\tau$ is in $B$, then either
\begin{enumerate}[(i)]
    \item $\tau^{-1}(y_0) \in R$; or
    \item $\tau^{-1}(y_0)=x_i$ for some $1 \leq i\leq d$, and at least $i$ of the vertices $y_1, \ldots, y_d$ are contained in $\tau(R \cup \{x_1, \ldots, x_{i-1}\})$.
\end{enumerate}
Note that condition (i) holds for $\tau=\sigma$.  Thus, it suffices to show that if this statement holds for some $\tau$ in $B$, then it also holds for any $\tau'$ that is obtained from $\tau$ by an $(X,Y)$-friendly swap.  Suppose $\tau'$ is obtained from $\tau$ by an $(X,Y)$-friendly swap across the edge $\{u,v\} \in E(X)$ (so that $\{\tau(u), \tau(v)\}\in E(Y)$).  We consider a number of case distinctions.
\begin{enumerate}[(a)]
    \item Suppose $\tau$ satisfies condition (i) and neither $u$ nor $v$ equals $\tau^{-1}(y_0)$.  Then $(\tau')^{-1}(y_0)=\tau^{-1}(y_0) \in R$, so $\tau'$ satisfies condition (i).
    \item Suppose $\tau$ satisfies condition (i) and $\tau^{-1}(y_0)\in\{u,v\}$. Without loss of generality, assume $\tau^{-1}(y_0)=u$.  Note that $\tau(v)=y_j$ for some $1 \leq j \leq d$ and that $(\tau')^{-1}(y_0)=v$.  If $v \in R$, then $\tau'$ satisfies condition (i).  Otherwise, $v=x_1$, and $\tau'$ satisfies condition (ii) because $y_j \in \tau'(R)$.
    \item Suppose $\tau$ satisfies condition (ii) for some $i$ and neither $u$ nor $v$ equals $\tau^{-1}(y_0)$.  Then $u$ and $v$ are contained in the same connected component of $X\vert_{V(X)\setminus\{x_i\}}$.  In particular, we see that
    $$\tau'(R \cup \{x_1, \ldots, x_{i-1}\})=\tau(R \cup \{x_1, \ldots, x_{i-1}\}),$$
    which shows that $\tau'$ satisfies condition (ii).
    \item Suppose $\tau$ satisfies condition (ii) for some $2 \leq i \leq d-1$ and $\tau^{-1}(y_0)\in\{u,v\}$. Without loss of generality, assume $u=\tau^{-1}(y_0)=x_i$.  Since the only neighbors of $x_i$ are $x_{i-1}$ and $x_{i+1}$, we conclude that $v$ is one of these two vertices.  Note (as before) that $\tau(v)=y_j$ for some $1 \leq j \leq d$ and that $(\tau')^{-1}(y_0)=v$.  If $v=x_{i-1}$, then $(\tau')^{-1}(y_0)=x_{i-1}$ and
    $$\tau'(R \cup \{x_1, \ldots, x_{i-2}\})=\tau(R \cup \{x_1, \ldots, x_{i-1}\}) \setminus \{y_j\}.$$
    The assumption on $\tau$ implies that this set has size at least $i-1$, so $\tau'$ satisfies condition (ii).  If instead $v=x_{i+1}$, then
    $(\tau')^{-1}(y_0)=x_{i+1}$ and
    $$\tau'(R \cup \{x_1, \ldots, x_i\})=\tau(R \cup \{x_1, \ldots, x_{i-1}\}) \cup \{y_j\}.$$
    The assumption on $\tau$ implies that this set has size at least $i+1$, so $\tau'$ satisfies condition (ii).
    \item Suppose that $d=1$, that $\tau$ satisfies condition (ii) (necessarily for $i=1$), and that $\tau^{-1}(y_0) \in \{u,v\}$.  Without loss of generality, assume $u=\tau^{-1}(y_0)=x_1$.  Note that $v=\tau^{-1}(y_1)$, where $\tau^{-1}(y_1) \in R$ by assumption.  Then $(\tau')^{-1}(y_0)=v \in R$, so $\tau'$ satisfies condition (i).
    \item Suppose that $d>1$, that $\tau$ satisfies condition (ii) for $i=1$, and that $\tau^{-1}(y_0) \in \{u,v\}$.  Without loss of generality, assume $u=\tau^{-1}(y_0)=x_1$.  Note that we have $\tau(v)=y_j$ for some $1 \leq j \leq d$.  Recall that $x_1$ is adjacent to only $x_2$ and some vertices in $R$, so these are the only possible vertices that $v$ can be.  If $v=x_2$, then $(\tau')^{-1}(y_0)=x_2$ and
    $$\tau'(R \cup \{x_1\})=\tau(R) \cup \{y_j\}.$$
    The assumption on $\tau$ implies that this set has size at least $2$, so $\tau'$ satisfies condition (ii).  If $v \in R$, then $(\tau')^{-1}(y_0)=v \in R$, which shows that $\tau'$ satisfies condition (i).
    \item  Suppose that $d>1$, that $\tau$ satisfies condition (ii) for $i=d$, and that $\tau^{-1}(y_0)\in\{u,v\}$. Without loss of generality, assume $u=\tau^{-1}(y_0)=x_d$.  Note that we have $v=\tau^{-1}(y_j)$ for some $1 \leq j \leq d$.  All vertices $y_1, \ldots, y_d$ are contained in $\tau(R \cup \{x_1, \ldots, x_{d-1}\})$ by assumption, so we see that $v \in R \cup \{x_1, \ldots, x_{d-1}\}$.  Since $v$ is adjacent to $u$, we conclude that $v=x_{d-1}$.  Then $(\tau')^{-1}(y_0)=x_{d-1}$ and
    $$\tau'(R \cup \{x_1, \ldots, x_{d-2}\})=\tau(R \cup \{x_1, \ldots, x_{d-1}\}) \setminus \{y_j\}.$$
    The assumption on $\tau$ implies that this set has size at least $d-2$, so $\tau'$ satisfies condition (ii).
\end{enumerate}
This analysis exhausts the possible cases and completes the proof.
\end{proof}

\begin{figure}[ht]
\begin{center}
\includegraphics[height=5cm]{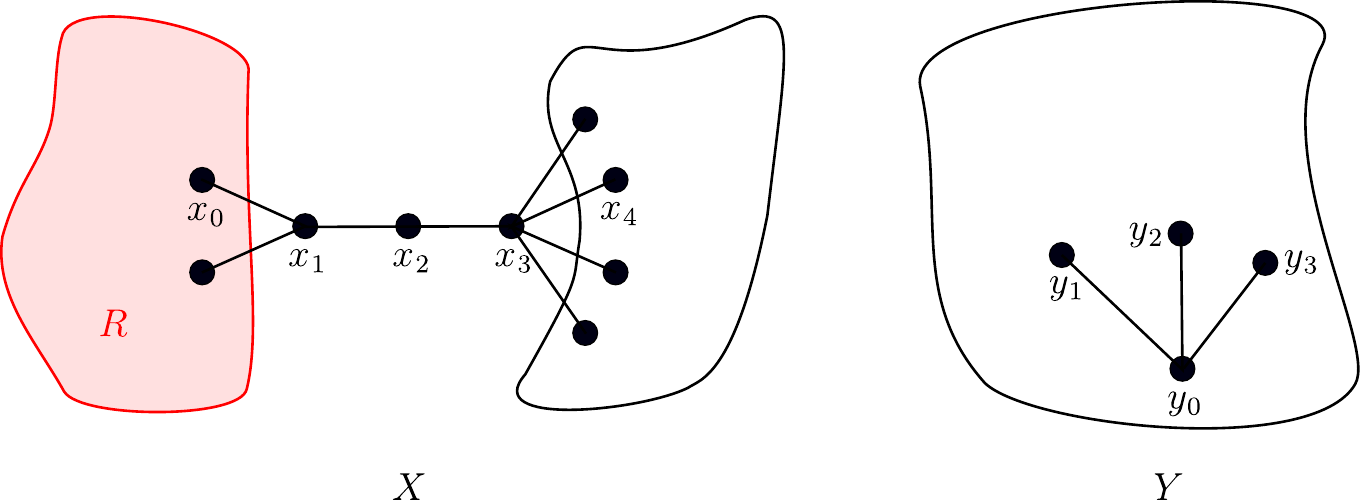}
\caption{A schematic illustration of the proof of Theorem~\ref{thm:cut-path} (with $d=3$).}
\label{Fig6}
\end{center}  
\end{figure}

We now derive some consequences of Theorem~\ref{thm:cut-path} regarding specific families of graphs. The \dfn{lollipop graph} $\mathsf{Lollipop}_{k,m}$ is the graph with vertex set $[k+m]$ and edge set \[E(\mathsf{Lollipop}_{k,m})=\{\{i,i+1\}:1\leq i\leq k\}\cup\{\{i,j\}:k+1\leq i<j\leq k+m\}.\] In other words, $\mathsf{Lollipop}_{k,m}$ is obtained by identifying the vertex $k+1$ in the path graph $\mathsf{Path}_{k+1}$ with the vertex $k+1$ in the complete graph on the vertex set $\{k+1,\ldots,k+m\}$. The following corollary is immediate from Theorem~\ref{thm:cut-path}. 

\begin{corollary}\label{Thm6}
Let $X$ be a subgraph of $\mathsf{Lollipop}_{n-m,m}$. If $Y$ is a graph on $n$ vertices such that $\FS(X,Y)$ is connected, then the minimum degree of $Y$ is at least $n-m+1$.
\end{corollary}

The graph $\mathsf{Lollipop}_{n-3,3}$ is a prolongation of $K_3$. Therefore, Corollary~\ref{cor:triangle-orb} guarantees that $\FS(\mathsf{Lollipop}_{n-3,3},Y)$ is connected whenever $Y$ is a graph on $n$ vertices with minimum degree at least $n-2$. Consequently, we can appeal to Corollary~\ref{Thm6} to obtain a complete classification of the graphs $Y$ such that $\FS(\mathsf{Lollipop}_{n-3,3})$ is connected. 

\begin{corollary}\label{Cor4}
Let $Y$ be a graph on $n$ vertices. The graph $\FS(\mathsf{Lollipop}_{n-3,3},Y)$ is connected if and only if the minimum degree of $Y$ is at least $n-2$. 
\end{corollary}

\begin{remark}\label{Rem:Lollipop}
The graph $\mathsf{Lollipop}_{n-5,5}$ is a prolongation of the graph $X$ in Corollary~\ref{cor:weird-1}, so it follows from that corollary that $\FS(\mathsf{Lollipop}_{n-5,5},Y)$ is connected whenever $Y$ is a graph on $n$ vertices with minimum degree at least $n-3$. On the other hand, Corollary~\ref{Thm6} tells us that $\FS(\mathsf{Lollipop}_{n-5,5},Y)$ is disconnected whenever $Y$ is a graph on $n$ vertices with minimum degree at most $n-5$. We leave open the characterization of graphs $Y$ such that $\FS(\mathsf{Lollipop}_{n-5,5},Y)$ is connected; in order to finish the characterization, it suffices to consider only the graphs $Y$ with minimum degree $n-4$. 
\end{remark}

We end this section with a discussion of one additional nice family of graphs. For $n\geq 3$, let $\mathsf{D}_n$ be the graph with vertex set $[n]$ and edge set $\{\{i,i+1\}:1\leq i\leq n-2\}\cup\{\{n-2,n\}\}$. The notation $\mathsf D_n$ is chosen because this graph is the Dynkin diagram of type $D_n$. We will show that for $n\geq 4$, the characterization of the graphs $Y$ such that $\FS(\mathsf{D}_n,Y)$ is connected is precisely the same as the characterization of the graphs $Y$ such that $\FS(\mathsf{Lollipop}_{n-3,3},Y)$ is connected. Hence, the edge $\{n-1,n\}$ in $\mathsf{Lollipop}_{n-3,3}$ has no influence on whether or not $\FS(\mathsf{Lollipop}_{n-3,3},Y)$ is connected. 

\begin{theorem}\label{Thm7}
Let $Y$ be a graph on $n\geq 4$ vertices. The graph $\FS(\mathsf{D}_n,Y)$ is connected if and only if the minimum degree of $Y$ is at least $n-2$. 
\end{theorem}

Before proving this result, we establish a simple lemma. Recall that $\alpha_{\overline Y}(\lambda)$ denotes the acyclic orientation of $\overline Y$ associated to the permutation $\lambda$, where an edge $\{y,z\}$ is oriented from $y$ to $z$ if and only if $\lambda^{-1}(y)<\lambda^{-1}(z)$.

\begin{lemma}\label{Lem1}
Let $Y$ be a graph with vertex set $[n]$. If $\lambda$ and $\lambda'$ are two permutations in $\mathfrak S_n$ such that $\lambda(n)=\lambda'(n)$ and $\alpha_{\overline Y}(\lambda)=\alpha_{\overline Y}(\lambda')$, then $\lambda$ and $\lambda'$ are in the same connected component of $\FS(\mathsf D_n,Y)$.
\end{lemma}

\begin{proof}
Let $Y'=Y\vert_{[n]\setminus\{\lambda(n)\}}$. The restrictions $\lambda\vert_{[n-1]}$ and $\lambda'\vert_{[n-1]}$ are two vertices of $\FS(\mathsf{Path}_{n-1},Y')$ such that $\alpha_{\overline {Y'}}(\lambda\vert_{[n-1]})=\alpha_{\overline {Y'}}(\lambda'\vert_{[n-1]})$, so it follows immediately from Theorem~\ref{Thm1} that there is a sequence of $(\mathsf{Path}_{n-1},Y')$-friendly swaps transforming $\lambda\vert_{[n-1]}$ into $\lambda'\vert_{[n-1]}$. Because $\mathsf{Path}_{n-1}$ is a subgraph of $\mathsf D_n$ and $Y'$ is a subgraph of $Y$, this same sequence of friendly swaps can be viewed as a sequence of $(\mathsf D_n,Y)$-friendly swaps that transforms $\lambda$ into $\lambda'$.  
\end{proof}

\begin{proof}[Proof of Theorem~\ref{Thm7}]
We may assume $V(Y)=[n]$. The graph $\mathsf{D_n}$ is a subgraph of $\mathsf{Lollipop}_{n-3,3}$. Therefore, if the minimum degree of $Y$ is at most $n-3$, then it follows from Proposition~\ref{Prop1} and Corollary~\ref{Cor4} that $\FS(\mathsf{D}_n,Y)$ is disconnected. 

Now assume $Y$ has minimum degree at least $n-2$. This means that $\overline Y$ has maximum degree at most $1$. In other words, $\overline Y$ is a disjoint union of copies of $K_1$ and $K_2$. Choose a permutation $\sigma\in\mathfrak S_n$, and let $B$ be the connected component of $\FS(\mathsf D_n,Y)$ containing $\sigma$. We are going to show that the identity permutation in $\mathfrak S_n$ is also in $B$. Since $\sigma$ was arbitrary, this will imply that $\FS(\mathsf D_n,Y)$ is connected. Consider the acyclic orientation $\alpha_{\overline Y}(\sigma)$ of $\overline Y$ associated to $\sigma$. Suppose $\{u,v\}$ is an edge of $\overline Y$ with $\sigma^{-1}(u)<\sigma^{-1}(v)$. This means that in $\alpha_{\overline Y}(\sigma)$, the edge $\{u,v\}$ is oriented from $u$ to $v$. Since $\{u,v\}$ is an isolated edge in $\overline Y$, the vertex $v$ must be a sink of $\alpha_{\overline Y}(\sigma)$. We claim that there is a permutation $\widehat\sigma$ in $B$ such that $\alpha_{\overline Y}(\widehat\sigma)$ is the same as $\alpha_{\overline Y}(\sigma)$ except with the orientation of $\{u,v\}$ reversed. First, suppose $\sigma(n)=v$. Let $r$ be such that $\sigma(n-r)=u$. We can apply $(\mathsf D_n,Y)$-friendly swaps across $\{n-r,n-r+1\},\{n-r+1,n-r+2\},\ldots,\{n-2,n-1\}$ (in this order) and then another $(\mathsf{D}_n,Y)$-friendly swap across $\{n-2,n\}$ so as to obtain the permutation \[\widehat\sigma=\sigma\circ(n-r\,\,n-r+1)\circ(n-r+1\,\, n-r+2)\circ\cdots\circ(n-2\,\,n-1)\circ(n-2\,\, n).\] It is straightforward to check that $\widehat \sigma$ is in $B$ and that $\alpha_{\overline Y}(\widehat\sigma)$ is obtained from $\alpha_{\overline Y}(\sigma)$ by reversing the orientation of $\{u,v\}$. Next, assume that $\sigma(n)\neq v$. Since $u$ is a source of $\alpha_{\overline Y}(\sigma)$, we also have $\sigma(n)\neq u$. There exists a permutation $\sigma'\in\mathfrak S_n$ such that $\alpha_{\overline Y}(\sigma')=\alpha_{\overline Y}(\sigma)$, $\sigma'(n)=\sigma(n)$, $\sigma'(n-1)=v$, and $\sigma'(n-2)=u$. Applying Lemma~\ref{Lem1} with $\lambda=\sigma$ and $\lambda'=\sigma'$, we find that $\sigma'$ is in $B$. If we apply a $(\mathsf D_n,Y)$-friendly swap across $\{n-2,n\}$ to $\sigma'$, we obtain a permutation $\widehat\sigma$ in $B$ such that the acyclic orientations $\alpha_{\overline Y}(\sigma)$ and $\alpha_{\overline Y}(\widehat\sigma)$ agree except in the orientation of $\{u,v\}$, as desired. 

We have shown that we can use a sequence of $(\mathsf D_n,Y)$-friendly swaps to reverse the orientation of a single edge in the acyclic orientation of $\overline Y$ associated to a permutation. By repeating this argument, we can eventually construct a permutation $\tau$ in $B$ such that $\alpha_{\overline Y}(\tau)$ is any prescribed acyclic orientation of $\overline Y$. In particular, we can choose $\tau$ in $B$ so that each edge $\{a,b\}$ in $\overline Y$ with $a<b$ is oriented from $a$ to $b$ in $\alpha_{\overline Y}(\tau)$. In other words, $\alpha_{\overline Y}(\tau)=\beta$, where $\beta$ is the acyclic orientation of $\overline Y$ associated to the identity permutation in $\mathfrak S_n$.  

We are going to show that there is a permutation $\tau'$ in $B$ with $\alpha_{\overline Y}(\tau')=\beta$ and $\tau'(n)=n$. If we can do this, then it will follow from Lemma~\ref{Lem1} (with $\lambda=\tau'$ and $\lambda'$ the identity permutation) that the identity permutation is also in $B$, as claimed. If $\tau(n)=n$, then we can simply set $\tau'=\tau$. Thus, let us assume that $\tau(n)\neq n$. Using the fact that $n\geq 4$, it is not difficult to check that there exists a linear extension $\widehat\tau$ of $\beta$ such that $\widehat\tau(n)=\tau(n)$ and $\widehat\tau(n-2)=n$. Invoking Lemma~\ref{Lem1} with $\lambda=\tau$ and $\lambda'=\widehat\tau$, we see that $\widehat\tau$ is in $B$. Note that $n$ is a sink of $\beta$ by the definition of $\beta$. Since $\widehat\tau$ is a linear extension of $\beta$, the vertex $\widehat\tau(n)$ must also be a sink of $\beta$. The vertices $\widehat\tau(n)$ and $n$ are distinct sinks of $\beta$, so they are not adjacent in $\overline Y$. Therefore, we can perform a $(\mathsf D_n,Y)$-friendly swap across $\{n-2,n\}$ in order to transform $\widehat\tau$ into a different linear extension $\tau'$ of $\beta$ with $\tau'(n)=n$. This permutation $\tau'$ is in $B$, as desired.   
\end{proof}

\section{Concluding Remarks and Open Problems}\label{Sec:Conclusion}

We conclude with several open problems and suggestions for future inquiries along the lines of the present paper.

\subsection{Other graphs}

In Sections~\ref{Sec:Paths} and \ref{Sec:Cycles}, we gained a full understanding of the connected components of $\FS(\mathsf{Path}_n,Y)$ and $\FS(\mathsf{Cycle}_n,Y)$. It could be interesting consider complements of paths and cycles by investigating the connected components of $\FS(\overline{\mathsf{Path}_n},Y)$ and $\FS(\overline{\mathsf{Cycle}_n},Y)$. Another natural direction would be the exploration of graphs of the form $\FS(K_{k,n-k},Y)$, where $K_{k,n-k}$ is the complete bipartite graph with partite sets of sizes $k$ and $n-k$. Note that $K_{1,n-1}$ is isomorphic to $\mathsf{Star}_n$, which was studied thoroughly by Wilson \cite{wilson}. It might be interesting to consider just the specific complete bipartite graphs $K_{2,n-2}$, or possibly $K_{r,r}$. It would also be interesting to obtain more general results about the graphs $\FS(X,Y)$ when $X$ is a tree (or even a specific type of tree). Let us also recall Remark~\ref{Rem:Lollipop}, which asks for a characterization of the graphs $Y$ such that $\FS(\mathsf{Lollipop}_{n-5,5},Y)$ is connected. Of course, several other nice families of graphs that we have not mentioned could also give rise to interesting results.

\subsection{Making $\FS(X,Y)$ connected for all reasonable graphs $Y$}

One might naturally ask for the sparsest graph $Y$ on $n$ vertices such that $\FS(X,Y)$ is connected whenever $X$ is a connected graph on $n$ vertices. For $\FS(\mathsf{Path}_n,Y)$ to be connected, $Y$ must be a complete graph. A less trivial variant of this problem asks for the sparsest graph $Y$ on $n$ vertices such that $\FS(X,Y)$ is connected whenever $X$ is a biconnected graph on $n$ vertices. In this case, $\FS(\mathsf{Cycle}_n,Y)$ must be connected, so it follows from Corollary~\ref{Cor2} that $\overline Y$ must be a forest consisting of trees of coprime sizes. This implies that $Y$ must have at least $\binom{n}{2}-(n-2)$ edges. 

This bound of $\binom{n}{2}-(n-2)$ is tight. Indeed, suppose $Y$ is a graph such that $\overline Y$ is the disjoint union of a tree on $n-1$ vertices and an isolated vertex $v^*$. Then $Y$ contains a subgraph isomorphic to $\mathsf{Star}_n$ (with $v^*$ as the center of the star). Suppose $X$ is biconnected. We want to show that $\FS(X,Y)$ is connected. If $X$ is isomorphic to $\mathsf{Cycle}_n$, then the connectedness of $\FS(X,Y)$ follows from Corollary~\ref{Cor2}. If $n=7$ and $X$ is isomorphic to the graph $\theta_0$ from Theorem~\ref{ThmWilson}, then one can check by computer that $\FS(X,Y)$ is connected. Now assume $X$ is not isomorphic to $\theta_0$ or a cycle graph. If $X$ is not bipartite, then $\FS(X,Y)$ is connected because $\FS(X,\mathsf{Star}_n)$ is connected by Wilson's theorem ($\FS(X,\mathsf{Star}_n)$ and $\FS(\mathsf{Star}_n,X)$ are isomorphic). If $X$ is bipartite, then the connectedness of $\FS(X,Y)$ follows from Remark~\ref{Rem1} because $\mathsf{Star}_n$ is isomorphic to a proper subgraph of $Y$. 

We have shown that if $\FS(X,Y)$ is connected for all biconnected graphs $X$, then $\overline Y$ must a forest consisting of trees of coprime sizes. We have seen that some of these choices for $Y$ do indeed make $\FS(X,Y)$ connected for all biconnected $X$. The following conjecture states that \emph{all} such choices of $Y$ satisfy this property. 

\begin{conjecture}
Let $Y$ be a graph on $n$ vertices such that $\overline Y$ is a forest consisting of trees $\mathscr T_1,\ldots,\mathscr T_r$ such that $\gcd(|V(\mathscr T_1)|,\ldots,|V(\mathscr T_r)|)=1$. If $X$ is a biconnected graph on $n$ vertices, then $\FS(X,Y)$ is connected. 
\end{conjecture}

\subsection{Diameter and girth}

In a different direction, one might ask about the diameter and girth of the graph $\FS(X,Y)$ or its connected components.  Recall that if $X$ and $Y$ are graphs on $n$ vertices, then $\FS(X,Y)$ has $n!$ vertices.  Must the diameter of a connected component of $\FS(X,Y)$ be polynomially bounded in $n$?  What changes if we restrict our attention to $X$ and $Y$ such that $\FS(X,Y)$ is connected? For the girth of $\FS(X,Y)$, we can say somewhat more.  Since $\FS(X,Y)$ is bipartite (by Proposition~\ref{Prop2}), its girth is either even (and greater than or equal to $4$) or infinite.  It is easy to see that if $X$ and $Y$ each have $2$ pairwise disjoint edges, then the girth of $\FS(X,Y)$ is exactly $4$.  The behavior in the remaining case (which can be reduced to the setting where $X$ is a star and $Y$ is connected) remains unknown.

It would also be interesting to understand the diameters of connected components of some specific graphs that we have studied already, such as $\FS(\mathsf{Path}_n,Y)$ and $\FS(\mathsf{Cycle}_n,Y)$. For example, Proposition~\ref{PropToggles} states that for every $n$-element poset $P$, the group of bijections generated by the toggle operators $t_1,\ldots,t_{n-1}$ acts transitively on the set $\mathcal L(P)$ of linear extensions of $P$. It is natural to ask for the maximum distance between two linear extensions of $P$, where the distance is measured by the number of toggle operators needed to change the first linear extension into the second. To state this problem in the language we have been using throughout this article, let $Z$ be the Hasse diagram of $P$ (viewed as a graph), and let $Y=\overline Z$. Consider the acyclic orientation $\alpha$ of $\overline Y=Z$ given by orienting each edge $\{x,y\}$ from $x$ to $y$ whenever $y$ covers $x$ in $P$. Then we are asking for the diameter of the connected component $H_\alpha$ of $\FS(\mathsf{Path}_n,Y)$ as defined in Theorem~\ref{Thm1}. 

\subsection{New equivalence relations for acyclic orientations}

Equivalences of acyclic orientations under various flips played a major role in our analysis in Section~\ref{Sec:Cycles}.  Fix nonnegative integers $a$ and $b$ and a graph $G$, and consider the set $\Acyc(G)$ of acyclic orientations of $G$.  Suppose that in an acyclic orientation $\alpha \in \Acyc(G)$, there are pairwise non-adjacent vertices $u_1, \ldots, u_a$ and $v_1, \ldots, v_b$ such that every $u_i$ is a source and every $v_i$ is a sink.  Then we can obtain a new acyclic orientation by reversing the directions of all edges incident to any of these vertices; equivalently, we simultaneously flip every $v_i$ into a sink and every $u_i$ into a source.  Call such an operation an \dfn{$(a,b)$-flip}.  We say that acyclic orientations $\alpha, \alpha'$ are \dfn{$\{a,b\}$-flip equivalent} if $\alpha'$ can be obtained from $\alpha$ by a sequence of $(a,b)$-flips and $(b,a)$-flips.  (The reader can easily verify that this is in fact an equivalence relation.)  In this language, the toric acyclic orientations discussed thoroughly in Section~\ref{Sec:Cycles} (which have received significant attention in previous articles such as \cite{Develin}) are simply $\{0,1\}$-flip equivalence classes. Furthermore, the double-flip equivalence classes introduced in Section~\ref{Sec:Cycles}, which are used to parameterize the connected components of graphs of the form $\FS(\mathsf{Cycle}_n,Y)$, are simply $\{1,1\}$-flip equivalence class. We also saw that there is a nontrivial connection between $\{0,1\}$-flip equivalence classes and $\{1,1\}$-flip equivalence classes: if $G$ has connected components of sizes $n_1,\ldots,n_r$ and $\nu=\gcd(n_1,\ldots,n_r)$, then every $\{0,1\}$-flip equivalence class is the union of exactly $\nu$ $\{1,1\}$-flip equivalence classes. Furthermore, we saw that for each fixed $\{0,1\}$-flip equivalence class $[\alpha]_\sim$ of $\Acyc(G)$, the $\nu$ $\{1,1\}$-flip equivalence classes contained in $[\alpha]_\sim$ all have the same number of linear extensions. 

What can be said about $\{a,b\}$-flip equivalence classes for other choices of $\{a,b\}$?  Are there other interesting connections between $\{a,b\}$-flip equivalence classes and $\{a',b'\}$-flip equivalence classes for different choices of $a,a',b,b'$? 

\subsection{Probabilistic and extremal questions}
It is natural to ask about the connectedness of the friends-and-strangers graph $\FS(X,Y)$ that results from choosing $X$ and $Y$ to be independent Erd\H{o}s-R\'enyi random graphs in $\mathcal G(n,p)$. In a separate article with Noga Alon \cite{Typical}, we prove that the threshold probability $p$ where $\FS(X,Y)$ changes from being disconnected with high probability to being connected with high probability is $p=n^{-1/2+o(1)}$. 

In the same article, we obtain estimates for the minimum $d$ such that whenever $X$ and $Y$ are $n$-vertex graphs with minimum degrees at least $d$, the friends-and-strangers graph $\FS(X,Y)$ is connected. 

\subsection{Right versus left multiplication}

One of the referees suggested looking at analogues of friends-and-strangers graphs for other Coxeter groups. The construction discussed in this section provides one such analogue that works for any group. 

Suppose $\Gamma$ is a group and $S\subseteq\Gamma$ is closed under taking inverses, meaning $\tau^{-1}\in S$ for every $\tau\in S$. The \dfn{right Cayley graph}\footnote{This deviates slightly from standard terminology since we do not require $S$ to generate $\Gamma$.} $\mathsf{Cay}_{\mathsf R}(\Gamma,S)$ is the graph with vertex set $\Gamma$ and edge set \[E(\mathsf{Cay}_{\mathsf R}(\Gamma,S))=\{\{\sigma,\sigma\tau\}:\sigma\in\Gamma,\tau\in S\}.\] In other words, edges in $\mathsf{Cay}_{\mathsf R}(\Gamma,S)$ correspond to right multiplication by elements of $S$. 
Similarly, the \dfn{left Cayley graph} $\mathsf{Cay}_{\mathsf L}(\Gamma,S)$ is the graph with vertex set $\mathfrak S_n$ and edge set \[E(\mathsf{Cay}_{\mathsf L}(\Gamma,S))=\{\{\sigma,\tau\sigma\}:\sigma\in\Gamma,\tau\in S\}.\] 

For $i,j\in[n]$, let us identify the pair $\{i,j\}$ with the transposition $(i\,j)$ in $\mathfrak S_n$ that swaps $i$ and $j$. If $X$ and $Y$ are graphs with $V(X)=V(Y)=[n]$, then the edge set of the friends-and-strangers graph of $X$ and $Y$ is \[E(\FS(X,Y))=E(\mathsf{Cay}_{\mathsf R}(\mathfrak S_n,E(X)))\cap E(\mathsf{Cay}_{\mathsf L}(\mathfrak S_n,E(Y))).\] In particular, we have  $\FS(X,K_n)=\mathsf{Cay}_{\mathsf R}(\mathfrak S_n,E(X))$ and $\FS(K_n,Y)=\mathsf{Cay}_{\mathsf L}(\mathfrak S_n,E(Y))$. This suggests a natural generalization of friends-and-strangers graphs. Namely, if $S_{\mathsf R}$ and $S_{\mathsf L}$ are subsets of a group $\Gamma$ that are each closed under taking inverses, then we can consider the graph with vertex set $\Gamma$ and edge set \[E(\mathsf{Cay}_{\mathsf R}(\Gamma,S_{\mathsf R}))\cap E(\mathsf{Cay}_{\mathsf L}(\Gamma,S_{\mathsf L})).\] 

The preceding definition seems far too broad for one to say anything meaningful in full generality; however, it is possible that there could be interesting directions to explore for specific groups (for instance, following the referee's suggestion, hyperoctahedral groups). Even when $\Gamma=\mathfrak S_n$, the above definition still vastly generalizes friends-and-strangers graphs because we do not require the elements of $S_{\mathsf R}$ and $S_{\mathsf L}$ to be transpositions. It could be fruitful to investigate whether these graphs have interesting structural properties when $\Gamma=\mathfrak S_n$ and the sets $S_{\mathsf R}$ and $S_{\mathsf L}$ are chosen very specifically.

\section*{Acknowledgments}
We thank Christian Gaetz for pointing us to the article \cite{stanley} and Alexander Postnikov for showing us that the connected components of $\FS(\mathsf{Path}_n,Y)$ correspond to the acyclic orientations of $\overline Y$. We thank Noga Alon and Matthew Macauley for engaging in very helpful conversations. We also thank the anonymous referees for their thorough reading and detailed comments. We are especially grateful to the referee who outlined their own very detailed alternative view on friends-and-strangers graphs. The first author was supported by a Fannie and John Hertz Foundation Fellowship and an NSF Graduate Research Fellowship (grant number DGE-1656466).

\end{document}